\documentclass[a4paper,11pt,twoside]{article}
\usepackage[left=2cm,right=2cm,top=2cm,bottom=3cm]{geometry}

\usepackage[T1]{fontenc}
\usepackage[utf8]{inputenc} 
\usepackage[english]{babel} 

\usepackage[super]{nth}  
\usepackage{units}
\usepackage{amsmath}

\usepackage{comment}
\usepackage{xcolor}
\definecolor{forestgreen}{rgb}{0.0, 0.5, 0.0}

\usepackage[usestackEOL]{stackengine}
\usepackage{amsmath} 
\usepackage{amssymb} 
\usepackage{amsthm}
\usepackage{amsfonts}
\usepackage{mathtools}
\usepackage{mathrsfs}
\usepackage{braket} 
\usepackage{graphicx}
\usepackage{graphics}
\usepackage{stmaryrd}
\usepackage{textcomp}
\usepackage[export]{adjustbox}
\usepackage{cancel}
\usepackage{mathabx,epsfig}
\usepackage{cases}
\usepackage[overload]{empheq}
\mathtoolsset{showonlyrefs=true,showmanualtags=true}

\usepackage[all]{xy}

\newcommand{\numberset}{\mathbb}
\newcommand{\N}{\numberset{N}}

\newcommand{\R}{\numberset{R}}

\newcommand{\E}{\numberset{E}}
\newcommand{\F}{\mathcal{F}}  
\DeclareMathOperator{\law}{\text{Law}}  
\DeclareMathOperator{\w}{\mathcal{W}}  
\DeclareMathOperator{\pr}{\mathscr{P}} 
\DeclareMathOperator{\p}{\numberset{P}}  
\DeclareMathOperator{\meas}{meas}   
\newcommand{\ndelta}{M}    
\newcommand{\ngrid}{N}  
\DeclareMathOperator{\linfproctwo}{L^{\infty}(Q;H_T^2)}   
\DeclareMathOperator{\linfmeastwo}{L^{\infty}(Q;C_T^2)} 

\theoremstyle{definition}
\newtheorem{definition}{Definition}[section]
\newtheorem{theorem}[definition]{Theorem}
\newtheorem{lemma}[definition]{Lemma}

\newtheorem{remark}[definition]{Remark}

\newtheorem*{notation*}{Notation}
\newtheorem*{definition*}{Definition}
\newtheorem*{theorem*}{Theorem}
\newtheorem*{lemma*}{Lemma}
\newtheorem*{corollary*}{Corollary}
\newtheorem*{proposition*}{Proposition}
\newtheorem*{fact*}{Fact}
\newtheorem*{example*}{Example}
\newtheorem*{claim*}{Claim}
\newtheorem*{remark*}{Remark}
\newtheorem*{conjecture*}{Conjecture}

\numberwithin{equation}{section}

\allowdisplaybreaks[4]

\title{The mean field limit of stochastic differential equation systems modelling grid cells}
\author{José A. Carrillo\thanks{Mathematical Institute, University of Oxford, Oxford OX2 6GG, UK (carrillo@maths.ox.ac.uk)}
\and Andrea Clini\thanks{Mathematical Institute, University of Oxford, Oxford OX2 6GG, UK (andrea.clini@maths.ox.ac.uk)}
\and Susanne Solem \thanks{Department of Mathematics, University of Life Sciences, NO-1433 \AA s, Norway (susanne.solem@nmbu.no)}}
\date{\today}

\allowdisplaybreaks[4]

\begin{document}

\maketitle

\begin{abstract}
Several differential equation models have been proposed to explain the formation of patterns characteristic of the grid cell network. Understanding the robustness of these patterns with respect to noise is one of the key open questions in computational neuroscience. 
In the present work, we analyze a family of stochastic differential systems modelling grid cell networks. Furthermore, the well-posedness of the associated McKean--Vlasov and Fokker--Planck equations, describing the average behavior of the networks, is established.
Finally, we rigorously prove the mean field limit of these systems and provide a sharp rate of convergence for their empirical measures.
\end{abstract}




\section{Introduction}
\label{introduction}
The discovery of a type of neurons in the brain named grid cells in 2005 \cite{gridcells} led to a breakthrough in the understanding of the navigational system in mammalian brains, see \cite{McNaughtonMoser} for an extensive review. These neurons fire as an animal moves around in an open area, enabling the animal to understand its position in space. The grid cell network has commonly been described by deterministic continuous attractor network dynamics through a system of neural field models \cite{Ermentrout2010, McNaughtonetal, burakfiete, coueyetal}, which are based on the classical papers \cite{WC1, WC2, Amari1977}. The models can fairly accurately predict what can be observed in experiments. However, the question of how the grid cell network is affected by noise, posed as a challenge in \cite{tenyears}, has been left open. 

In \cite{BurakFieteNoise} fundamental limits on how information dissipates in attractor networks of noisy neurons were derived. A different direction pursuing further understanding of the effect of noise on grid cell networks was made in \cite{CHS2020} by studying a system of Fokker--Planck-like partial differential equations (PDEs). The system of PDEs was derived by adding noise to the attractor network models in \cite{burakfiete, coueyetal} and formally taking the mean field limit. In the present manuscript this limit is rigorously proved. In addition, we derive the limit for more general noise terms, which covers the models considered in \cite{BurakFieteNoise, AB20}.  

The mean field limit of interacting particle systems has lately received lots of attention in mathematical biology
\cite{BCC, FI15, CS2018}, see \cite{JW17} for a survey. The closest result to the analysis presented in this work, shows the mean field limit of a stochastic delayed set of interacting neurons \cite{TouboulPhysD}. The system of stochastic differential equations (SDEs) describing interacting grid cells in this work introduces different challenges: boundary conditions imposing positivity of the activity level of the neurons, non-linearity of the firing rate, and coupling between different families of neurons.  

The neural model under consideration, which is based on the model in \cite{burakfiete}, can be described as follows. Given space points $x_1,\dots,x_N\in Q$ in a region $Q$ of the neural cortex, we will consider the following model for the interaction among $NM$ neurons stacked in $N$ columns at locations $x_i$ with $M$ neurons each, where $u_{ik}^\beta$ represents the activity level with orientation $\beta$ of the $k^{th}$ neuron at location $x_i$:
\begin{subequations}
\label{concrete model with reflection term}
\begin{align}[left ={\empheqlbrace}]
        u^\beta_{ik}(t)\tau_i^\beta=&\, \tau^\beta_iu_{ik}^\beta(0)+\sigma W_{ik}^\beta(t)-\ell^\beta_{ik}(t)
        \nonumber \\
        & +\int_0^t\left(-u_{ik}^\beta(r)+\phi\Big(B^\beta(x_i,r)+\frac{1}{4\ngrid\ndelta}\sum_{\gamma=1}^4\sum_{j=1}^{\ngrid}\sum_{m=1}^\ndelta K^\gamma(x_i-x_j)u_{jm}^\gamma(r)\Big)\right)\,dr,
        \label{concrete model particle system} \\
        \ell_{ik}^{\beta}(t)= &\, -\big|\ell_{ik}^{\beta}\big|(t),\quad  \big|\ell_{ik}^{\beta}\big|(t)=\int_0^t1_{\{u_{ik}^\beta(r)=0\}}d\big|\ell_{ik}^{\beta}\big|(r)\quad\text{for }\beta=1,2,3,4.
    \label{reflection term concrete model particle system}
  \end{align}
  \end{subequations}
For simplicity, we consider $Q=[0,1]^d$.
The results in this work are easily extended to any bounded open subset $Q\subseteq\R^d$, for any $d\geq1$.
Here, for integers $k=1,\dots,M$, we have i.i.d. families of random initial conditions $\{u_{ik}(0)\}_{i=1,\dots,N}$ for each space point $x_i$ in the cortex $Q$.
Moreover, for integers $i=1,\dots, N$ and $k=1,\dots, M$, we have $4$-dimensional Brownian motions $(W_{ik}^\beta)_{\beta=1,2,3,4}$, which can also be correlated.

The nonlinear function $\phi:\R\to\R$, representing the firing rate of neurons in the network, is globally Lipschitz, whereas the external inputs $B^\beta:Q\times\R\to\R$ and the interaction kernels $K^\beta:\R^{d_Q}\to\R$ for $\beta=1,2,3,4$ are only required to be locally bounded functions and $\alpha$-H\"older continuous in the $x$ variable for some $\alpha\in(0,1]$. The interaction kernels takes into account the inhibitory/excitatory effect on nearby neurons. A typical choice of the interaction kernel in computational neuroscience \cite{burakfiete} is given by the so-called Mexican hat function. The relaxation times $\tau_{i}^\beta$ satisfy the condition $0<\inf_{i,\beta}\tau_i^\beta\leq\sup_{i,\beta}\tau_i^\beta<+\infty$. 

Finally, for each $i$, $k$ and $\beta$, the term $\ell_{ik}^\beta$ is a finite variation process defined by \eqref{reflection term concrete model particle system} which prevents the activity level $u_{ik}^\beta$ from taking negative values. 
Namely, as we can see in its definition, at each time $t$ this process equals the opposite of its total variation $\ell_{ik}^{\beta}(t)=-\big|\ell_{ik}^{\beta}\big|(t)$. In turn, the total variation stays constant when $u_{ik}^\beta>0$ and it increases in the form $\big|\ell_{ik}^{\beta}\big|(t)=\int_0^t1_{\{u_{ik}^\beta(r)=0\}}d\big|\ell_{ik}^{\beta}\big|(r)$ when $u_{ik}^\beta=0$, so as to push  $u_{ik}^\beta$ away from zero which is being dragged by the other terms at the right hand side of \eqref{concrete model particle system}.
The introduction of such terms and constraints is therefore known as imposing \emph{reflecting boundary conditions} and $\ell_{ik}^\beta$ is called a \emph{reflection term}.
The existence and uniqueness of such a term need of course to be proved and this process is often referred to as the \emph{Skorokhod problem}.
Precise details concerning the well-posedness and the construction of the reflection term in our setting are all presented in the seminal papers \cite{Lions-Sznitman-1984-SDEreflectingBC, Sznitman-1984-nonlinear-reflecting-diffusion} by Lions and Sznitman.

Going back to \eqref{concrete model particle system}, we notice that the argument of $\phi$ in \eqref{concrete model particle system} can be rewritten as
\begin{equation*}
    \frac{1}{4\ngrid\ndelta}\sum_{\gamma=1}^4\sum_{j=1}^{\ngrid}\sum_{m=1}^\ndelta K^\gamma(x_i-x_j)u_{jm}^\gamma(r)=\int_{Q\times\R^4}\frac{1}{4}\sum_{\gamma=1}^4K^\gamma(x_i-y)u^\gamma f_{\ngrid,\ndelta}(r,dy,du),
\end{equation*}
by considering the empirical measure associated to these particles, that is
\begin{equation}\label{empirical measure}
    f_{\ngrid,\ndelta}(r,dy,du)=\frac{1}{\ngrid\ndelta}\sum_{j=1}^{\ngrid}\sum_{m=1}^\ndelta\delta_{(x_j,u_{jm}(r))}\quad\text{regarded as a measure on $Q\times\R^4$.}
\end{equation}

Concerning the initial conditions and the form of the noise term in \eqref{concrete model particle system}, from a modelling point of view it is reasonable to assume that, for $k\in \N$, we have i.i.d. families of initial conditions $\big(u_k(x,0)\big)_{x\in Q}$ for each space point $x$ in the cortex $Q$.
Similarly, we assume that, for $k\in\N$, we have independent $4$-dimensional space-time white noise terms $\big(W_k(x,t)\big)_{t\geq0, x\in Q}$.
Naively, $W_k^\beta(x,t)$ is a centered Gaussian random field indexed by $k\in\N$, $\beta=1,2,3,4$, $x\in Q$ and $t\in[0,\infty)$ with covariance
\begin{equation}\label{naive construction white noise}
    \E\left[W_k^\beta(x,t)W_h^\gamma(y,s)\right]=(t\wedge s)\, \delta_0(k-h)\,\delta_0(\beta-\gamma)\,\delta_0(x-y).
\end{equation}
Then we can just choose points $x_1,\dots,x_N\in Q$ and set $u_{ik}(0)\coloneqq u_k(x_i,0)$ and $W_{ik}(t)\coloneqq W_k(x_i,t)$.
As long as we work in a countable setting, this naive construction can be made rigorous upon taking a suitable modification of the $W_{ik}$'s via the Kolmogorov continuity theorem.
We also point out that the way we choose the cloud of points $x_1,\dots x_N\in Q$ is not that important if we are only concerned with the discrete model for fixed $M$ and $N$.
However, to get a nice limiting behaviour as $N,M\to\infty$, it is useful to take these points to be the nodes of a grid of $Q$ whose mesh tends to zero.
Precise details on this are given in Section \ref{Comparison between the particle system and the limiting model}.

\begin{remark} \label{loss of indepence}
One should not expect the initial data $(u_k(x,0))_{x\in Q}$ to be independent for different values of $x$, nor to be equidistributed.
Indeed, from the point of view of modelling in neuroscience, $u_k(x,0)$ should be close to $u_k(y,0)$ for $x$ close to $y$.
This fact will have consequences both on the exchangeability properties of the particles $u_{ik}$, which are expected to be exchangeable in the index $k$ only, and on the rate of convergence towards the limiting behavior.
\end{remark}

As we let $M,N\to\infty$ the limiting behaviour should be described by independent copies, in the column index $k$, of solutions to an associated mean field McKean--Vlasov equation.
Namely, the activity level of any neuron located at a point $x\in Q$ should satisfy an equation like:
\begin{align}[left ={\empheqlbrace}]\label{concrete model McKean--Vlasov}
    \begin{split}
    \Bar{u}^{\beta}(x,t)\tau^\beta(x)=&\tau^\beta(x)u^\beta(x,0)+\sigma W^\beta(x,t)-\bar{\ell}^\beta(x,t)\\
    & +\int_0^t\Bigg(-\Bar{u}^\beta(x,r)+\phi\Big(B^\beta(x,r)+\frac{1}{4}\sum_{\gamma=1}^4\int_{Q\times\R^4}\!\!\!\!\!\!\!\!K^\gamma(x-y)u^\gamma f(r,y,du)dy\Big)\Bigg)\,dr,
    \\
    \bar{\ell}^{\beta}(x,t)=&-\big|\bar{\ell}^{\beta}(x,\cdot)\big|(t),\quad \big|\bar{\ell}^{\beta}(x,\cdot)\big|(t)=\int_0^t1_{\{\Bar{u}^\beta(x,r)=0\}}d\big|\bar{\ell}^{\beta}(x,\cdot)\big|(r)\quad\text{for }\beta=1,2,3,4,
    \end{split}
\end{align}
where we have set $f(t,y,du)\coloneqq\law_{\R^4}(\Bar{u}(y,t))$ considered as a measure on $\R^4$ depending on $t\in[0,\infty)$ and $y\in Q$.
Notice that in turn this induces a probability measure $f(t,dx,du)$ on $Q\times\R^4$ defined by integration as 
\begin{equation}\label{McKean--Vlasov law on Q times R^4}
    \int_{Q\times\R^4}\varphi(x,u) f(t,dx,du)\coloneqq\int_Q\int_{\R^4}\varphi(x,u)f(t,x,du)\,dx\qquad\text{for any $\varphi\in C_b(Q\times\R^4)$}.
\end{equation}
For each fixed $x\in Q$ and $\beta=1,2,3,4$, the finite variation process $\bar{\ell}^\beta(x,t)$ is again the reflection term coming from the Skorokhod problem (see the explanation after equation \eqref{concrete model with reflection term}) and it ensures that $\bar{u}^\beta(x,t)\geq0$ for every $x$, $t$ and $\beta$. We refer the reader to \cite{Sznitman-1984-nonlinear-reflecting-diffusion} for the details about such a process in the context of a classical McKean--Vlasov equation.

\begin{remark}\label{ill-posedeness of the uncorrelated mckean-vlasov equation}
The McKean--Vlasov equation \eqref{concrete model McKean--Vlasov} suffers from a major technical issue.
Indeed, formula \eqref{naive construction white noise} does define an $\R$-valued Gaussian random field. However, it is well-known that such a random field cannot be jointly measurable in the $x$ variable and the sample $\omega$.
This reflects into lack of $x$-measurability of the particles $\bar{u}^{\beta}(x,t)$ and, in turn, into that of the law $f(t,x,du)$, which we need to be Lebesgue integrable.
In this work, we resolve this issue by considering $\epsilon$-correlated noise.

Another approach, coming from the theory of mean field games, is to address the issue by introducing a ``Fubini extension'' of the product probability space $Q\times\Omega$.
We refer the reader to \cite{aurell2021stochastic} and the references therein.
However, this approach did not seem to fit our modelling purposes.
It allows to regain the $x$-measurability only with respect to a bigger $\sigma$-algebra, \emph{strictly} containing the Lebesgue measurable sets.
In turn, the space integral in \eqref{concrete model McKean--Vlasov} would not be taken with respect to the Lebesgue measure, but instead with respect to some exotic extension of this.
\end{remark}

A formal application of the It\^o formula shows that $f$, the joint distribution of the activity levels $u^\beta$ in the four directions $\beta$, satisfies the nonlinear Fokker--Planck equation
\begin{align}\label{concrete model nonlinear fokker--planck}
    \begin{split}
        \partial_tf(t,x,u)\!+\sum_{\beta=1}^4\frac{1}{\tau^\beta(x)}\partial_{u^\beta}\! \Bigg(f(t,x,u)\Big(\!\!-\!u^\beta\!\!+\!\phi\Big(B^\beta(x,t)\!+ \!\frac{1}{4}\!\!&\sum_{\gamma=1}^4\!\int_{Q\times\R^4}\!\!\!\!\!\!\!\!\!\!K^\gamma(x-y)v^\gamma f(t,y,dv)dy\Big)\,\Big)\Bigg)
        \\
        &=\frac{\sigma^2}2\sum_{\beta=1}^4\frac{1}{\tau^\beta(x)^2}\partial^2_{u^\beta u^\beta}f(t,x,u),
    \end{split}
\end{align}
in the weak sense, with initial condition $f(0,x,du)=\law_{\R^4}(u(x,0))$ and subjected to the no-flux boundary conditions, for $\beta=1,2,3,4$,
\begin{equation}
        \phi\Big(B^\beta(x,t)\!+\frac{1}{4}\sum_{\gamma=1}^4\!\int_{Q\times\R^4}\!\!\!\!\!\!\!\!\!K^\gamma(x-y)v^\gamma f(t,y,dv)dy\Big)f(t,x,u)-\frac{\sigma^2}2\frac{1}{\tau^\beta(x)}\frac{\partial}{\partial u^\beta}f(t,x,u)\Big|_{u^\beta=0}=0,
\end{equation}
which come from the reflecting boundary conditions at the SDE level.

\begin{remark}
It is worth pointing out that equation \eqref{concrete model nonlinear fokker--planck} would arise as the law of $\bar{u}(x,t)$ even if we set $W(x,t)\equiv B_t$ for every $x\in Q$ for a single Brownian motion $B_t$, that is if all the particles were affected by the same noise.
The same holds for many other choices of $W(x,t)$, and follows immediately from the It\^o formula: the effect of the term $W(x,t)$ is only to generate diffusion in the $u$ variable, for fixed $x$.
The choice of noise to consider in \eqref{concrete model particle system} and \eqref{concrete model McKean--Vlasov} is therefore dictated by modelling purposes only.
\end{remark}

\begin{remark}\label{marginals of fokker--planck}
We notice that for each $\beta=1,2,3,4$, integrating equation \eqref{concrete model nonlinear fokker--planck} in $\R_+^3$ over the remaining variables $u^\gamma$ for $\gamma\neq\beta$ and exploiting the boundary conditions, we get the equation satisfied by the marginal distribution $f^\beta(r,y,du^\beta)=\law_{\R}(\Bar{u}^\beta(y,r))$. Namely, we obtain
\begin{align}\label{concrete model marginal nonlinear fokker--planck}
    \begin{split}
        \partial_tf^\beta(t,x,u^\beta)\!+\!\frac{1}{\tau^\beta(x)}\partial_{u^\beta}\! \Bigg(\!\!f^\beta(t,x,u^\beta)\Big(\!\!-\!u^\beta\!\!+\!\phi\Big(\!B^\beta(x,t)\!+ \!\frac{1}{4}\!\!\sum_{\gamma=1}^4\!\int_{Q\times\R^4}\!\!\!\!\!\!\!\!\!\!&K^\gamma(x-y)v^\gamma f^\gamma(t,y,dv^\gamma)dy\Big)\Big)\Bigg)
        \\
        &=
        \frac{\sigma^2}2\frac{1}{\tau^\beta(x)^2}\frac{\partial^2 f^\beta}{(\partial u^\beta)^2}(t,x,u^\beta).
    \end{split}
\end{align}
In particular, we stress the fact that each marginal $f^\beta$ satisfies an equation involving only the other marginals $f^\gamma$, and not the full joint distribution $f$.
On the other hand, if we sum equation \eqref{concrete model marginal nonlinear fokker--planck} over $\beta=1,2,3,4$, then we get back equation \eqref{concrete model nonlinear fokker--planck} above for the decoupled distribution $\Tilde{f}\coloneqq\Pi_{\beta=1}^4f^\beta$.
Thus equation \eqref{concrete model nonlinear fokker--planck} and the system of equations \eqref{concrete model marginal nonlinear fokker--planck} for $\beta=1,2,3,4$ are completely equivalent, at least for decoupled initial data $f_0=\Pi_{\beta=1}^4f^\beta_0$.
Finally, Theorem \ref{Well-posedness of the non-linear fokker--planck equations} below asserts we have existence and uniqueness for equation \eqref{concrete model nonlinear fokker--planck}.
The previous argument then shows that, if we start with decoupled initial data, this structure is preserved: the corresponding solution satisfies $f(t)=\Pi_\beta f^\beta(t)$ for all $t\geq0$. Notice that \eqref{concrete model marginal nonlinear fokker--planck} is the model formally introduced in \cite{CHS2020}.
\end{remark}

The structure of this work is as follows. The next section is devoted to introduce the notation and the setting needed for the results. We finish the section by stating the main theorems concerning the mean field limit of \eqref{concrete model with reflection term} and its extensions. Sections \ref{Strong existence and uniqueness for the particle systems section} and \ref{the limiting model} focus on the existence and uniqueness of the particle systems, and the associated McKean--Vlasov equations and Fokker--Planck type PDEs. The main core of this work is found in Section \ref{Comparison between the particle system and the limiting model}, where we rigorously prove the mean field limit. Section \ref{Convergence of empirical measures} adapts previous results on empirical measure error estimates \cite{Fournier2013OnTR} to the present setting to provide rates of convergence for the associated empirical measure.


\section{Preliminaries and main results}

\subsection{Hypotheses and notation}
\label{Hypotheses and notations}

In this section we introduce the hypotheses we assume for our problem.
First, we point out that the results of this paper extend to the more general particle system
\begin{align}[left ={\empheqlbrace}]\label{abstract 4 model particle system}
\begin{split}
    u_{ik}(t)=&\,u_{ik}(0)+\int_0^t\!b(x_i,r,u_{ik}(r),f_{N,\ndelta}(r))\,dr+\int_0^t\!\sigma(x_i,r,u_{ik}(r),f_{N,\ndelta}(r))\,dW_{ik}(r)-\ell_{ik}(t),
    \\[4mm]
    \ell_{ik}^{\beta}(t)=&\,-|\ell_{ik}^{\beta}|(t),\quad |\ell_{ik}^{\beta}|(t)=\int_0^t1_{\{u^\beta_{ik}(r)=0\}}d|\ell_{ik}^{\beta}|(r)\quad\text{for }\beta=1,2,3,4,
\end{split}
\end{align}
for $N$ columns of $M$ neurons each, located at $x_1,\dots,x_N$, with general drift term $b$ and diffusion term $\sigma$.
Here $f_{N,\ndelta}(r,dy,du)$ is again the empirical measure associated to the particles \eqref{abstract 4 model particle system}, given by \eqref{empirical measure}. As before, $\ell^\beta_{ik}$ is the the reflection term coming from the Skorokhod problem \cite{Lions-Sznitman-1984-SDEreflectingBC} forcing $u_{ik}^\beta(t)\geq0$ for every $t\geq0$.

The precise details on the shape and hypotheses on $b$ and $\sigma$ are given here below and they are simply deduced from the properties of the concrete model \eqref{concrete model with reflection term}. 

Let $\pr(Q\times\R^4)$ denote the set of probability measures on $Q\times\R^4$, for $\beta=1,2,3,4$ we assume that $b_\beta,\sigma_{\beta}:Q\times\R^+\times\R^4\times\pr(Q\times\R^4)\to\R$ take the forms
\begin{align}\label{drift term structure}
   &b_{\beta}(x,r,u,f)=b^{\beta}_0(x,r,u)+\phi_{b_{\beta}}\left(\int_{Q\times\R^4}b_1^\beta(x,y,r,u,v)\,f(dy,dv)\right),
    \\
    \label{diffusion term structure}
    &\sigma_{\beta}(x,r,u,f)=\sigma^{\beta}_0(x,r,u)+\phi_{\sigma_{\beta}}\left(\int_{Q\times\R^4}\sigma_1^\beta(x,y,r,u,v)\,f(dy,dv)\right).
\end{align}
Having in mind the concrete model \eqref{concrete model particle system}, we suppose $b_0^\beta,\sigma_0^\beta:Q\times\R^+\times\R^4\to\R$ are measurable, locally bounded, Lipschitz in $u\in\R^4$ uniformly in $x,r\in Q\times\R^+$, and $\alpha$-H\"older in $x\in Q$ uniformly in $u,r\in \R^4\times\R^+$. That is
\begin{align}
    \label{b_0 sigma_0 lipschitz property}
    \big|b_0^{\beta}(x,r,u)-b_0^{\beta}(x',r,u')\big|+\big|\sigma_0^{\beta}(x,r,u)-\sigma_0^{\beta}(x',r,u')\big|&\leq L \left(|x-x'|^{\alpha}+ |u-u'|\right),
    \\
    \label{b_0 sigma_0 sublinear property}
    \big|b_0^{\beta}(x,r,u)\big|+\big|\sigma_0^{\beta}(x,r,u)\big|&\leq C \left(1+|u|\right),
\end{align}
for all $x,x',u,u',r\in Q^2\times\left(\R^4\right)^2\times\R^+$, for suitable constants $L$ and $C$.
Furthermore we take the functions $\phi_{b_{\beta}},\phi_{\sigma_{\beta}}:\R\to\R$ to be globally Lipschitz functions, and thus with sublinear growth.
Similarly, the mappings $b_1^\beta,\sigma_1^\beta:Q\times Q \times\R^+\times\R^4\times \R^4\to\R$ are measurable, locally bounded, Lipschitz in $u,v\in\R^4$ uniformly in $x,y,r\in Q^2\times\R^+$, and $\alpha$-H\"older in $x,y\in Q$ uniformly in $u,v,r\in \left(\R^4\right)^2\times\R^+$.
That is,
\begin{align}
    \label{b_1 sigma_1 lipschitz property}
    \big|b_1^{\beta}(x,y,r,u,v)-b_1^{\beta}(x,y,r,u',v')\big| \qquad & \nonumber
    \\+\big|\sigma_1^{\beta}(x,y,r,u,v)-\sigma_1^{\beta}(x,y,r,u',v')\big|&\leq L \left(|x-x'|^{\alpha}+|y-y'|^{\alpha}+|u-u'|+|v-v'|\right),
    \\
    \label{b_1 sigma_1 sublinear property}
    \big|b_1^{\beta}(x,y,r,u,v)\big|+\big|\sigma_1^{\beta}(x,y,r,u,v)\big|&\leq C \big(1+|u|+|v|\big),
\end{align}
for all $x,y,x',y',u,v,u',v',r\in Q^4\times\left(\R^4\right)^4\times \R^+$, for suitable constants $L$ and $C$.

\begin{remark}
With the notation just introduced, the starting model \eqref{concrete model particle system} is recovered by setting
\begin{equation}\nonumber
   \tau^\beta(x) b^\beta(x,r,u,f)\!=\!-u^\beta\!+\!\phi\left(\!B^\beta(x,r)\!+\!\frac{1}{4}\!\sum_{\gamma=1}^4\!\!\int_{Q\times\R^4}\!\!\!\!\!\!\!\!\!\!\!\!K^\gamma(x-y)u^\gamma f(dy,du)\right),
    \quad \text{and} \quad
    \tau^\beta(x)\sigma(x,r,u,f)\!\equiv \sigma.
\end{equation}
\end{remark}

We now consider the limiting McKean--Vlasov system.
Taking into account the measurability issues pointed out in Remark \ref{ill-posedeness of the uncorrelated mckean-vlasov equation}, we consider instead equation \eqref{concrete model McKean--Vlasov} with a suitably rescaled $\epsilon$-correlated noise, for some $\epsilon>0$.
In the setting of the general particle system \eqref{abstract 4 model particle system}, the equation reads:
\begin{align}[left ={\empheqlbrace}]\label{abstract 4 model McKean--Vlasov}
    \begin{split}
    \Bar{u}^{\epsilon}(x,t)=&u(x,0)+\int_0^tb(x,r,\Bar{u}^{\epsilon}(x,r),f(r))\,dr+\int_0^t\sigma(x,r,\Bar{u}^{\epsilon}(x,r),f(r))\,dW^{\epsilon}(x,r)-\bar{\ell}(x,t),
    \\
    \bar{\ell}^{\beta}(x,t)=&-|\bar{\ell}^{\beta}(x,\cdot)|(t),\quad |\bar{\ell}^{\beta}(x,\cdot)|(t)=\int_0^t1_{\{(\Bar{u}^{\epsilon})^\beta(x,r)=0\}}\,d|\bar{\ell}^{\beta}(x,\cdot)|(r)\quad\text{for }\beta=1,2,3,4,
    \end{split}
\end{align}
where $f(r,y,du)=\law_{\R^4}(\Bar{u}^{\epsilon}(y,r))$ is viewed as a measure on $\R^4$, and $f(r)=f(r,dx,du)$ the induced probability measure defined by \eqref{McKean--Vlasov law on Q times R^4} on $Q\times\R^4$.
Similarly, the reflection term $\bar{\ell}(x,t)$ still ensures $(\bar{u}^{\epsilon})^\beta(x,t)\geq0$ for each $x$, $t$ and $\beta$ (see again \cite{Sznitman-1984-nonlinear-reflecting-diffusion}).
Here $W^{\epsilon}:\Omega\times\R^d\times\R^+\to\R^4$ is a $4$-dimensional Gaussian random field with independent components $\beta=1,2,3,4$, zero mean and covariance 
\begin{equation}\label{formula/epsilon rescaled white noise}
    \E\left[W^{\epsilon,\beta}(x,t)W^{\epsilon,\beta}(y,s)\right]\!=\!(t\wedge s)\, C_{\rho}\,\epsilon^d\int_{\R^d}\rho_\epsilon(z-x)\rho_{\epsilon}(z-y)\,dz,\quad \text{for } C_{\rho}=\left(\int_{\R^d}\rho(z)^2\,dz\right)^{-1}\!\!,
\end{equation}
where $\rho:\R^d\to[0,1]$ is a radial mollifier supported in the unitary ball, and $\rho_\epsilon$ the $\epsilon$-rescaled version. Such a process $W^{\epsilon,\beta}$ can for example be obtained by convolution and rescaling from a ``mathematically rigorous'' space-time white noise (see e.g. \cite{daprato_zabczyk_1992}).
That is, a distribution valued process $W:\Omega\times\R^+\to \mathcal{S}'(\R^d)$ such that, for $\varphi\in\mathcal{S}(\R^d)$, the processes $\langle W_t,\varphi\rangle$ are jointly Gaussian with covariance 
\begin{equation}
    \E\left[\langle W_t,\varphi\rangle \,\langle W_s,\psi\rangle\right]=(t\wedge s)\int_{\R^d}\varphi(z)\,\psi(z)\,dz.
\end{equation}
Then, for $\beta=1,2,3,4$ and independent copies of such a white noise, one defines
\begin{equation}\label{formula/epsilon convoluted and rescaled noise}
    W^{\epsilon,\beta}(x,t)\coloneqq C_{\rho}^{\frac{1}{2}}\,\epsilon^{\frac{d}{2}}\,\langle W_t,\rho_{\epsilon}(\cdot-x)\rangle.
\end{equation}
For future reference, we highlight some of the properties of $W^\epsilon$.
First, from \eqref{formula/epsilon rescaled white noise} we have that $\E[W^{\epsilon,\beta}(x,t)W^{\epsilon,\gamma}(x,s)]=\delta_0(\beta-\gamma)\,t\wedge s$. 
Thus, for fixed $x$, the process $t\mapsto W^{\epsilon}(x,t)$ is a $4$-dimensional Brownian motion.
Similarly, from $\sup(\rho)\subseteq B(0,1)$ it follows that
\begin{equation}
    \E\left[W^{\epsilon}(x,t)W^{\epsilon}(y,s)\right]=0 \,\,\text{ if }\,\, |x-y|>2\epsilon.
\end{equation}
Hence the processes $W^{\epsilon}(x,t)$ and $W^{\epsilon}(y,t)$ are independent for $|x-y|>2\epsilon$.
Furthermore, using \eqref{formula/epsilon rescaled white noise} one computes
\begin{align}
    \E\left[\left|W^{\epsilon}(x,t)-W^{\epsilon}(y,s)\right|^2\right]
    &\leq C\,
    \E\left[\left|W^{\epsilon}(x,t)-W^{\epsilon}(x,s)\right|^2+\left|W^{\epsilon}(x,s)-W^{\epsilon}(y,s)\right|^2\right]
    \\
    &\leq C\left(
    |t-s|+s\, C_\rho\, \epsilon^d\int_{\R^d}\left(\rho_\epsilon(z-x)-\rho_{\epsilon}(z-y)\right)^2\,dz\right)
    \\
    &\leq C
    \left(|t-s|+\frac{|x-y|^2}{\epsilon^2}\right),
\end{align}
for a constant $C=C(\rho)$.
Similar estimates hold for any higher moment $p\geq 2$ and the Kolmogorov continuity theorem ensures the existence of a suitable modification of $W^\epsilon$ with continuous trajectories in both $x$ and $t$.
In particular, we have that $W^{\epsilon}$ is jointly measurable in $(x,t)\in\R^d\times\R^+$ and in the sample path $\omega\in\Omega$.
Finally, for any $x,y\in\R^d$, a direct computation shows that the quadratic variation of the martingale $W^{\epsilon,\beta}(x,t)-W^{\epsilon,\beta}(y,t)$ satisfies
\begin{gather}\label{formula/epsilon convoluted noise quadratic variation}
\begin{aligned}
    \left[W^{\epsilon,\beta}(x,\cdot)-W^{\epsilon,\beta}(y,\cdot)\right]_t&=t\,C_\rho\, \epsilon^d\int_{\R^d}\left(\rho_\epsilon(z-x)-\rho_{\epsilon}(z-y)\right)^2\,dz,
    \\
    &\leq t\,
    C\,\frac{|x-y|^2}{\epsilon^2},
\end{aligned}
\end{gather}
for a constant $C=C(\rho)$.

Finally, $f(r,y,du)=\law_{\R^4}(\Bar{u}^{\epsilon}(y,r))$ should solve the associated nonlinear Fokker--Planck equation with no-flux boundary conditions,
\begin{equation}\label{abstract 4 model nonlinear fokker--planck}
    \begin{cases}
        \displaystyle
        \partial_tf(t,x,u)\!+\!\!\nabla_u\!\cdot\!\Big(b(x,t,u,f(t))f(t,x,u)\Big)\!=\frac{1}{2}\sum_{\beta=1}^4\frac{\partial^2}{\partial u^\beta\partial u^\beta}\Big(\sigma_\beta(x,t,u,f(t))^2f(t,x,u)\Big),
        \\
        \displaystyle
        b_\beta(t,x,u,f(t))f(t,x,u)\!-\frac{1}{2}\frac{\partial}{\partial u^\beta}\Big(\sigma_{\beta}(x,t,u,f(t))^2f(t,x,u)\Big)~\Big|_{u^\beta=0}\!\!\!\!\!\!\!=0\!\!\quad\text{for $\beta=1,2,3,4$},
    \end{cases}
\end{equation}
in the weak sense.

Let us now see how \eqref{drift term structure}--\eqref{diffusion term structure} and the assumptions \eqref{b_0 sigma_0 lipschitz property}--\eqref{b_1 sigma_1 sublinear property} translate into H\"older, Lipschitz and sublinear growth properties of the actual drift and diffusion terms.
First, notice that for any fixed $f\in\pr(Q\times\R^4)$ the mappings
\begin{equation}\label{lipschitz and holder properties for f fixed}
Q\times\R^+\times\R^4\to\R\quad\big|\quad x,r,u\mapsto b_{\beta}(x,r,u,f),\,\sigma_{\beta}(x,r,u,f),
\end{equation}
are easily seen to be $\alpha$-H\"older in $x$, Lipschitz and with sublinear growth in $u$, uniformly in $r$. Next, given a Banach space $X$ with norm $|\cdot|_X$ and a positive integer $m\in\N$, let us denote by $\pr_m(X)$ the space of probability measures on $X$ with finite $m$th moments, endowed with the $m$th order Wasserstein distance (see e.g. \cite{Villani}),
\begin{equation}
    \label{wasserstein distance}
    \w_m(X)(P_1,P_2)\coloneqq\inf_{\pi\in\Pi(P_1,P_2)}\left\{\int_X|x-y|_X^m\,\pi(dx,dy)\right\}^{\frac{1}{m}},
\end{equation}
where $\Pi(P_1,P_2)$ denotes the set of probability measures on $X\times X$ with marginals $P_1$ and $P_2$.
When $X$ is clear from the context we write $\w_m(P_1,P_2)$.
Let $L^{\infty}(Q;\pr_m(X))$ be the space of measurable functions $f:Q\to\pr_m(X)$ such that $$\displaystyle|f|_{L^{\infty}(Q;\pr_m(X))}=\sup_{y\in Q}\Big(\int_X|x|^m\,f(y,dx)\Big)^{\frac{1}{m}}<\infty,$$ endowed with the distance
\begin{equation}\nonumber
    d_{L^{\infty}(Q;\pr_m(X))}(f,g)=\sup_{y\in Q}\w_m(X)\big(f(y,dx),g(y,dx)\big).
\end{equation}
Assume $f,g\in L^{\infty}(Q;\pr_m(\R^4))$. Then we can identify them as elements in 
$\pr(Q\times\R^4)$ by their actions on test functions
\begin{equation}\nonumber
    \int_{Q\times\R^4}\psi(y,s)\,f(dy,ds):=\int_{Q}\int_{\R^4}\psi(y,s)\,f(y,ds)\,dy\qquad\forall\psi\in C_b(Q\times\R^4),
\end{equation}
and similarly for $g$.
Now, using the structure \eqref{drift term structure}--\eqref{diffusion term structure}, the H\"older, Lipschitz and sublinear growth properties of $b_i^\beta$ and $\sigma_i^\beta$ for $i=0,1$, and H\"older's inequality, it is straightforward to prove the following lemma.

\begin{lemma}\label{b sigma lipschitz and sublinear properties if disintegration}
In the setting outlined above, and under the assumptions on $b$ and $\sigma$,
\begin{align}
    \big|b_{\beta}(x,r,u,f)\!-\!b_{\beta}(x',r,u',f')\big|+&\big|\sigma_{\beta}(x,r,u,f)\!-\!\sigma_{\beta}(x',r,u',f')\big|\!
    \\
    &\leq\! L \left(|x-x'|^{\alpha}+|u-u'|\!+\!d_{L^{\infty}(Q;\pr_m(\R^4))}(f,f')\right), \nonumber
    \\
    \big|b_{\beta}(x,r,u,f)\big|+\big|\sigma_{\beta}(x,r,u,f)\big|&\leq C \left(1+|u|+|f|_{L^{\infty}(Q;\pr_m(\R^4))}\right), \nonumber
\end{align}
for all $x,x',u,u',r\in Q^2\times\left(\R^4\right)^2\times\R^+$ and all $f,f'\in L^{\infty}(Q;\pr_m(\R^4))$, for suitable constants $L$ and $C$.
\end{lemma}

\subsection{Main results} 
\label{main results}

We now present the main results of this work.
The theorems are stated for the general models \eqref{abstract 4 model particle system}, \eqref{abstract 4 model McKean--Vlasov}, and \eqref{abstract 4 model nonlinear fokker--planck}.
First we present a result on existence and uniqueness of the particle systems, which is proved in Section \ref{Strong existence and uniqueness for the particle systems section}.

\begin{theorem}[Strong existence and uniqueness for the particle systems]\label{Strong existence and uniqueness for the particle systems}
Assume that the initial data satisfies $\sup_{1\leq i\leq \ngrid}\sup_{1\leq k\leq \ndelta}\E[|u_{ik}(0)|^{2}]<+\infty$. Then, under assumptions \eqref{drift term structure}--\eqref{b_1 sigma_1 sublinear property} on the coefficients, there exists a pathwise unique solution of the particle system \eqref{abstract 4 model particle system}.
\end{theorem}

Next we state the theorems on well-posedness of the McKean--Vlasov equations and the associated PDE.
The following two results are proved in Section \ref{the limiting model}.

\begin{theorem}[Strong existence and uniqueness of the McKean--Vlasov equation]\label{Strong existence and uniqueness for the McKean--Vlasov}
Under the assumptions \eqref{drift term structure}--\eqref{b_1 sigma_1 sublinear property} on the coefficients, for any initial data $u(\cdot,0)\in L^{\infty}(Q;L^2(\Omega))$,  and for any $\epsilon>0$, there exists a pathwise unique solution $u^{\epsilon}\in L^{\infty}(Q;L^2(\Omega;C([0,T];\R^4)))$ of the McKean--Vlasov equation \eqref{abstract 4 model McKean--Vlasov}.
Moreover, for every $T<\infty$,
\begin{equation}
\label{Strong existence and uniqueness for the McKean--Vlasov eq 1}
    \sup_{x\in Q}\E\left[\sup_{t\in[0,T]}|u^\epsilon(x,t)|^{2}\right]
    \leq
    C\left(1+\sup_{x\in Q}\E[|u(x,0)|^{2}]\right),
\end{equation}
for a constant $C = C(T,b,\sigma)$.
 Finally, if the initial data satisfies $u(\cdot,0)\in C^{\alpha}(Q;L^2(\Omega))$ for some $\alpha\in(0,1)$, then $u^{\epsilon}\in C^{\alpha}(Q;L^2(\Omega;C([0,T];\R^4)))$.
\end{theorem}

\begin{theorem}[Well-posedness of the non-linear Fokker--Planck equation]\label{Well-posedness of the non-linear fokker--planck equations}
Under the assumptions \eqref{drift term structure}--\eqref{b_1 sigma_1 sublinear property} on the coefficients,  for any initial data $f_0(x,du)\in L^{\infty}(Q;\pr_2(\R^4))$, there exists a weak solution $f(x,t,du)\in L^{\infty}(Q;C([0,\infty);\pr_2(\R^4)))$ of the non-linear Fokker--Planck equation \eqref{abstract 4 model nonlinear fokker--planck}.
If $|\sigma(x,t,u,g)|\geq c>0$ for every $x,t,u$ and $g$, the solution is also unique.
 The map $f$ is uniquely characterized as $f(t,x,du)=\law_{\R^4}\left(\bar{u}^{\epsilon}(x,t)\right)$ for any arbitrary $\epsilon>0$.
Moreover, for each fixed $x\in Q$ and for any time $T>0$, the restriction $f(x,t,du)|_{t\in[0,T]}$ can be seen as a probability measure on the space $C([0,T];\R^4)$ of continuous paths, and it satisfies 
\begin{equation}\label{Well-posedness of the non-linear fokker--planck equations eq 1}
    \sup_{x\in Q}\int_{C([0,T];\R^4)}\sup_{t\in[0,T]}|v(t)|^{2}\,f(x,dv)
    \leq
    C\left(1+\sup_{x\in Q}\E[|u(x,0)|^{2}]\right),
\end{equation}
for a constant $C = C(T,b,\sigma)$, where $v\in C([0,T];\R^4)$.
 Finally, if $f_0\in C^{\alpha}(Q;\pr_2(\R^4))$, then $f\in C^{\alpha}\left(Q;\pr_2\left(C([0,T];\R^4)\right)\right)$, that is to say
\begin{equation}\label{Well-posedness of the non-linear fokker--planck equations eq 2}
    \w_2\left(C([0,T];\R^4)\right)(f(x,\cdot),f(y,\cdot))\leq C\,|x-y|^\alpha\quad\forall x,y\in Q,
\end{equation}
for a constant $C=C(T,b,\sigma,f_0)$.
\end{theorem}

We finally present two statements concerning the convergence of the particle system towards the limiting model as $M, N\to\infty$.
The setting is the following.
For $k\in\N$, let $(W_k(x,t))_{x\in Q,t\geq 0}$ be independent $4$-dimensional space-time white noise terms  over $Q\times[0,\infty)$, which we then convolve and rescale to obtain $W_k^{\epsilon}$ as in formula \eqref{formula/epsilon convoluted and rescaled noise}.
 Similarly, let $u_k(\cdot,0)\in C^{\alpha}(Q;L^2(\Omega))$ be i.i.d. families of random initial conditions along the cortex $Q$, and let them be independent of all the white noise terms.
Finally, let $X_i$ for $i=1,\dots,N$ be points on a equispaced grid on $Q$, with squares of sidelength $N^{\frac{1}{d}}$.
More details on the setting and the proofs of the results are given in Sections \ref{Comparison between the particle system and the limiting model} and \ref{Convergence of empirical measures}.

\begin{theorem}[Mean squared error estimates for actual particles vs. McKean--Vlasov particles]\label{Mean squared error estimates for actual particles vs McKean--Vlasov particles}
In the setting  outlined above and in Theorems \ref{Strong existence and uniqueness for the particle systems} and \ref{Strong existence and uniqueness for the McKean--Vlasov}, for any $N,\ndelta\in\N$, let $u_{ik}^{\epsilon}(\cdot)$ be the solution of the particle system \eqref{abstract 4 model particle system}, with initial data $u_{ik}(0)\coloneqq u_k(X_i,0)$ and noise terms $W_{ik}(t)\coloneqq W_k^{\epsilon}(X_i,t)$. 
For each $k\in\N$, let $\Bar{u}_k^{\epsilon}(x,t)$ be the solution of the McKean--Vlasov equation \eqref{abstract 4 model McKean--Vlasov}, with initial data $(u_k(x,0))_{x\in Q}$ and rescaled $\epsilon$-correlated noise $(W_k^{\epsilon}(x,t))_{x\in Q,t\geq 0}$.
For each $i\in\N$, denote $\Bar{u}_{ik}^{\epsilon}(t)\coloneqq \Bar{u}_k^{\epsilon}(X_i,t)$.
Then, for any $T>0$,
\begin{equation}\label{Mean squared error estimates for actual particles vs McKean--Vlasov particles eq 1}
    \E\left[\,\sup_{r\in[0,t]}|u_{ik}^{\epsilon}(r)-\Bar{u}_{ik}^{\epsilon}(r)|^{2}\right]^{\nicefrac{1}{2}}\leq C \,t\,\left(\frac{1}{\ngrid^{\frac{\alpha}{d}}}+\frac{1}{\ndelta^{\frac{1}{2}}}\right)\left(1+\sup_{x\in Q}\E\left[|u_k(x,0)|^{2}\right]^{\nicefrac{1}{2}}\right),
\end{equation}
for any $i=1,\dots,\ngrid$, $k=1,\dots,\ndelta$ and $t\in[0,T]$, where $C = C(T,\rho,b,\sigma,[u(\cdot,0)]_{\alpha})$ and $[u(\cdot,0)]_{\alpha}$ denotes the H\"older seminorm of $u(\cdot,0)$.
\end{theorem}

We notice that the decay has rate $ \left(\unitfrac{1}{N^{\frac{\alpha}{d}}}+\unitfrac{1}{M^{\frac{1}{2}}}\right)$ instead of the usual $\unitfrac{1}{(MN)^{\unitfrac[]{1}{2}}}$, as we might expect according to classical results in mean field theory \cite{Sznitman1991} since we have $MN$ particles. 
As anticipated in Remark \ref{loss of indepence}, this phenomenon goes back to the fact that the particles $u_{ik}$ are exchangeable in the second index only. Hence, what we will get is a mean field limit in the column index $k$, but a Riemann sum type convergence in the space index $i$.
This phenomenon will be made clear when we perform the computations in Section \ref{Comparison between the particle system and the limiting model}.

We also remark that the ratio between $\epsilon$ and $N^{-\nicefrac{1}{d}}$ in Theorem \ref{Mean squared error estimates for actual particles vs McKean--Vlasov particles}, that is between the correlation radius of the noise and the distance among the neuron locations $x_i$, is completely arbitrary and the decay rate in \eqref{Mean squared error estimates for actual particles vs McKean--Vlasov particles eq 1} is independent of this.
The choice of this ratio is purely dictated by modelling arguments, namely by the correlation strength we want for the noise sensed by two nearby neurons, which for example can be taken to be zero.

Finally we translate the previous result about convergence of particles to the level of laws.

\begin{theorem}[Rate of convergence for the empirical measure]\label{Rate of convergence for empirical measures}
In the setting of Theorem \ref{Mean squared error estimates for actual particles vs McKean--Vlasov particles}, let $$f^{\epsilon}_{\ngrid,\ndelta}(t,dx,du)=\frac{1}{MN}\sum_{j=1}^N\sum_{m=1}^M\delta_{\big(X_j,u_{jm}^{\epsilon}(t)\big)}$$ be the empirical measure on $Q\times\R^4$ associated with the particle system \eqref{abstract 4 model particle system}.
Let $f(t,x,du)$ be the unique solution of the Fokker--Planck equation \eqref{abstract 4 model nonlinear fokker--planck} and consider the induced probability measure $f(t,dx,du)$ on $Q\times\R^4$ given by \eqref{McKean--Vlasov law on Q times R^4}.
Then, as $M,N\to\infty$, and possibly but not necessarily as $\epsilon\to0$, $f^{\epsilon}_{\ngrid,\ndelta}(t,dx,du)$ converges to $f(t,dx,du)$ in the Wasserstein distance in the sense
\begin{equation}
    \sup_{t\in[0,T]}\E\left[\w_1(Q\times\R^4)(f^{\epsilon}_{\ngrid,\ndelta}(t),f(t))\right]
    \leq
     C\left(1+\sup_{x\in Q}\E\left[|u_k(x,0)|^{2}\right]^{\frac{1}{2}}\right)\left(\frac{1}{\ngrid^{\frac{\alpha}{d}}}+\frac{1}{\ndelta^{\frac{1}{2}}}+\frac{1}{M^{\frac{1}{4}}}\right),
\end{equation}
for any $T>0$, for $C=C(T,\rho,b,\sigma,[u(\cdot,0)]_{\alpha},Q)$.
\end{theorem}


\section{Strong existence and uniqueness for the particle systems}
\label{Strong existence and uniqueness for the particle systems section}

In this section we establish strong existence and uniqueness for the particle system \eqref{abstract 4 model particle system}, thus proving Theorem \ref{Strong existence and uniqueness for the particle systems}.
The proof is based on a classical contraction argument and a crucial observation about the reflection term $\ell_{ik}$.

\begin{proof}[\textbf{Proof of Theorem \ref{Strong existence and uniqueness for the particle systems}.}]
\noindent
Fix $\ngrid,\ndelta\in\N$.
Take a probability space $(\Omega, \F, \p)$ supporting the initial conditions $u_{ik}(0):\Omega\to\R^4$ and the $4$-dimensional Brownian motions $W_{ik}(t)$ for $i=1,\dots,\ngrid$ and $k=1,\dots,\ndelta$.
Suppose $\E\left[|u_{ik}(0)|^2\right]<\infty$ for all $i$ and $k$.
For any $T>0$ let us define the Banach space 
\begin{equation}
    \label{continuous second moment processes}
    H_T^2\coloneqq\left\{\text{continuous adapted processes $Y_t:\Omega\to\R^4\quad$ with $\quad \E\left[\,\sup_{t \in[0,T]}|Y_t|^2\right]^{\frac{1}{2}}<\infty$}\right\}, 
\end{equation}
endowed with the norm $\|Y_{\cdot}\|\coloneqq\E\big[\sup_{t\in[0,T]}|Y_t|^2\big]^{\frac{1}{2}}$, and then consider the product space $(H_T^2)^{\ngrid\ndelta}$ equipped with the product norm.

Define $F:(H_T^2)^{\ngrid\ndelta}\to(H_T^2)^{\ngrid\ndelta}$ by sending an element $v_{ik}\in(H_T^2)^{\ngrid\ndelta}$ to the pathwise solutions $\Tilde{v}_{ik}$ of the SDEs with reflecting boundary conditions, for $i=1,\dots,\ngrid$ and $k=1,\dots,\ndelta$,
\begin{align}[left ={\empheqlbrace}]\label{strong existence and uniqueness 2nd moments proof 1}
    \begin{split}
    \Tilde{v}_{ik}(t)= &\, u_{ik}(0)+\int_0^tb(x_i,r,v_{ik}(r),f^v_{\ngrid,\ndelta}(r))\,dr+\int_0^t\sigma(x_i,r,v_{ik}(r),f^v_{\ngrid,\ndelta}(r))\,dW_{ik}(r)-\ell^v_{ik}(t),
    \\
    \ell_{ik}^{v,\beta}(t)= & -|\ell_{ik}^{v,\beta}|(t),\quad |\ell_{ik}^{v,\beta}|(t)=\int_0^t1_{\{\Tilde{v}_{ik}(r)=0\}}d|\ell_{ik}^{v,\beta}|(r)\quad\text{for }\beta=1,2,3,4.
    \end{split}
\end{align}
where we define $$f^v_{\ngrid,\ndelta}(t)=\frac{1}{\ngrid\ndelta}\sum_{j=1}^{\ngrid}\sum_{m=1}^\ndelta\delta_{\left(x_j,v_{jm}(t)\right)}.$$

Under the hypotheses \eqref{drift term structure}--\eqref{b_1 sigma_1 sublinear property} on $b$ and $\sigma$, and by straightforward modifications of the setting and the proofs in \cite{Sznitman-1984-nonlinear-reflecting-diffusion, Lions-Sznitman-1984-SDEreflectingBC}, strong existence and uniqueness can be established for the SDEs \eqref{strong existence and uniqueness 2nd moments proof 1} with initial data with bounded second moments. Moreover, for initial data with $\E\big[|u_{ik}(0)|^2\big]<\infty$ and data $v_{ik}\in H_T^2$, the solutions $\Tilde{v}_{ik}$ belong to $H_T^2$.

We want to find $T$ small enough so that the map $F$ is a contraction.
Take two elements $u_{ik},v_{ik}\in (H_T^2)^{\ngrid\ndelta}$, and consider $\Tilde{u}_{ik}=F(u_{ik})$ and $\Tilde{v}_{ik}=F(v_{ik})$.
We apply It\^o formula to $|\Tilde{u}_{ik}-\Tilde{v}_{ik}|^2$ and exploit the respective equations \eqref{strong existence and uniqueness 2nd moments proof 1} to get
\begin{gather}\label{strong existence and uniqueness 2nd moments proof 2}
\begin{aligned}
|\Tilde{u}_{ik}(t)-&\Tilde{v}_{ik}(t)|^2 \\
= & \,
2\int_0^t(\Tilde{u}_{ik}(r)-\Tilde{v}_{ik}(r))\big(b(x_i,r,u_{ik}(r),f^u_{\ngrid,\ndelta}(r))-b(x_i,r,v_{ik}(r),f^v_{\ngrid,\ndelta}(r))\big)\,dr
\\
&+2\int_0^t(\Tilde{u}_{ik}(r)-\Tilde{v}_{ik}(r))\big(\sigma(x_i,r,u_{ik}(r),f^u_{\ngrid,\ndelta}(r))-\sigma(x_i,r,v_{ik}(r),f^v_{\ngrid,\ndelta}(r))\big)\,dW_{ik}(r)
\\
&+ 2\int_0^t(\Tilde{u}_{ik}-\Tilde{v}_{ik})\big(d\ell_{ik}^{v}(r)-d\ell_{ik}^{u}(r)\big)
+\int_0^t\big(\sigma(x_i,r,u_{ik},f^u_{\ngrid,\ndelta})-\sigma(x_i,r,v_{ik},f^v_{\ngrid,\ndelta})\big)^2\,dr,
\end{aligned}
\end{gather}
Exploiting the very definition of the reflection terms $\ell_{ik}^u$ and $\ell_{ik}^v$ shows that the third term on the right hand side of \eqref{strong existence and uniqueness 2nd moments proof 2} is negative.
Indeed, we use the second line in \eqref{strong existence and uniqueness 2nd moments proof 1} to expand this term as
\begin{gather}\label{reflection term is negative}
\begin{aligned}
    \int_0^t(\Tilde{u}_{ik}-\Tilde{v}_{ik})\big(d\ell_{ik}^{v}(r)-d\ell_{ik}^{u}(r)\big)
    =
    \sum_{\beta=1}^{4}
    \bigg(
    &
    \int_0^t(\Tilde{u}_{ik}^{\beta}-\Tilde{v}_{ik}^{\beta})1_{\{\Tilde{u}_{ik}(r)=0\}}d|\ell_{ik}^{u,\beta}|(r)
    \\
    &+
    \int_0^t(\Tilde{v}_{ik}^{\beta}-\Tilde{u}_{ik}^{\beta})1_{\{\Tilde{v}_{ik}(r)=0\}}d|\ell_{ik}^{v,\beta}|(r)\bigg).
\end{aligned}
\end{gather}
Since the reflecting boundary conditions ensure that $\Tilde{u}_{ik}^{\beta},\Tilde{v}_{ik}^{\beta}\geq0$, we see that all the integrals in the sum on the right hand side are negative, since the integrands are.

Now we drop the third term in \eqref{strong existence and uniqueness 2nd moments proof 2}, take the supremum in time and apply the expectation to get
\begin{gather}
\label{strong existence and uniqueness 2nd moments proof 3}
\begin{aligned}
\E\Bigg[\sup_{t\in[0,T]}& |\Tilde{u}_{ik}(t)-\Tilde{v}_{ik}(t)|^2\Bigg]
\\
\leq
\Bigg(&\int_0^T\E\big[\left|\Tilde{u}_{ik}-\Tilde{v}_{ik}\right|\left|b(x_i,r,u_{ik}(r),f^u_{\ngrid,\ndelta}(r))-b(x_i,r,v_{ik}(r),f^v_{\ngrid,\ndelta}(r))\right|\big]\,dr
\\
&+\!\E\bigg[\!\!\sup_{t\in[0,T]}\!\!\Big(\!\int_0^t\!(\Tilde{u}_{ik}-\Tilde{v}_{ik})\big(\sigma(x_i,r,u_{ik}(r),f^u_{\ngrid,\ndelta}(r))\!-\!\sigma(x_i,r,v_{ik}(r),f^v_{\ngrid,\ndelta}(r))\big)\,dW_{ik}(r)\Big)\bigg]
\\
&+\int_0^T\E\Big[\left|\sigma(x_i,r,u_{ik}(r),f^u_{\ngrid,\ndelta}(r))-\sigma(x_i,r,v_{ik}(r),f^v_{\ngrid,\ndelta}(r))\right|^2\Big]\,dr
\,\Bigg).
\end{aligned}
\end{gather}
The second term on the right hand side is handled with the Burkholder--Davis--Gundy inequality and with H\"older's inequality:
\begin{gather}
\label{strong existence and uniqueness 2nd moments proof 4}
\begin{aligned}
\E\bigg[\!&\sup_{t\in[0,T]}\!\Big(\!\int_0^t\!(\Tilde{u}_{ik}-\Tilde{v}_{ik})\big(\sigma(x_i,r,u_{ik}(r),f^u_{\ngrid,\ndelta}(r))-\sigma(x_i,r,v_{ik}(r),f^v_{\ngrid,\ndelta}(r))\big)\,dW_{ik}(r)\Big)\bigg]
\\
&\leq
\E\bigg[\Big(\!\int_0^T\!\left|\Tilde{u}_{ik}-\Tilde{v}_{ik}\right|^2\left|\sigma(x_i,r,u_{ik}(r),f^u_{\ngrid,\ndelta}(r))-\sigma(x_i,r,v_{ik}(r),f^v_{\ngrid,\ndelta}(r))\right|^2\,dr\Big)^{\frac{1}{2}}\bigg]
\\
&\leq
\E\bigg[\!\sup_{t\in[0,T]}\!\left|\Tilde{u}_{ik}-\Tilde{v}_{ik}\right|\!\Big(\int_0^T\!\left|\sigma(x_i,r,u_{ik}(r),f^u_{\ngrid,\ndelta}(r))-\sigma(x_i,r,v_{ik}(r),f^v_{\ngrid,\ndelta}(r))\right|^2\,dr\Big)^{\frac{1}{2}}\bigg]
\\
&\leq
\frac{1}{2}
\E\bigg[\!\sup_{t\in[0,T]}\!\left|\Tilde{u}_{ik}-\Tilde{v}_{ik}\right|^2\bigg]
+\frac{1}{2}
\E\bigg[\!\int_0^T\!\left|\sigma(x_i,r,u_{ik}(r),f^u_{\ngrid,\ndelta}(r))-\sigma(x_i,r,v_{ik}(r),f^v_{\ngrid,\ndelta}(r))\right|^2\,dr\bigg].
\end{aligned}
\end{gather}
Then we absorb the first term on the right hand side of \eqref{strong existence and uniqueness 2nd moments proof 4} into the left hand side of \eqref{strong existence and uniqueness 2nd moments proof 3} to get, for $C$ a numeric constant,
\begin{gather}
\label{strong existence and uniqueness 2nd moments proof 5}
\begin{aligned}
\E\Bigg[\!\!\sup_{t\in[0,T]}\! |\Tilde{u}_{ik}(t)&-\Tilde{v}_{ik}(t)|^2\Bigg]
\\
\leq\,C \bigg(\!&\int_0^T\!\!\E\big[\left|\Tilde{u}_{ik}-\Tilde{v}_{ik}\right|\left|b(x_i,r,u_{ik}(r),f^u_{\ngrid,\ndelta}(r))-b(x_i,r,v_{ik}(r),f^v_{\ngrid,\ndelta}(r))\right|\big]\,dr
\\
&+\int_0^T\E\Big[\left|\sigma(x_i,r,u_{ik}(r),f^u_{\ngrid,\ndelta}(r))-\sigma(x_i,r,v_{ik}(r),f^v_{\ngrid,\ndelta}(r))\right|^2\Big]\,dr\bigg).
\end{aligned}
\end{gather}

Now, we use the structure \eqref{drift term structure}--\eqref{diffusion term structure} and the Lipschitz properties \eqref{b_0 sigma_0 lipschitz property}--\eqref{b_1 sigma_1 lipschitz property} of $b$ and $\sigma$, the definition of $f_{MN}^u$ and $f_{MN}^v$, and applications of H\"older's inequality to get, for $C=C(b,\sigma)$: 
\begin{align}\label{strong existence and uniqueness 2nd moments proof 6}
\E\left[\sup_{t\in[0,T]}|\Tilde{u}_{ik}(t)-\Tilde{v}_{ik}(t)|^2\right]\leq
C\bigg(& \int_0^T\E\left[|\Tilde{u}_{ik}(r)-\Tilde{v}_{ik}(r)|^2\right] \,dr
\\
&+ \frac{1}{MN}\sum_{j=1}^{\ngrid}\sum_{m=1}^\ndelta \int_0^T\E\left[|u_{jm}(r)-v_{jm}(r)|^2\right] \,dr \bigg).
\end{align}
Then, we exploit Gr\"onwall's lemma to get, for $C=C(T,b,\sigma)$:
\begin{gather*}\label{strong existence and uniqueness 2nd moments proof 4}
\begin{aligned}
\E\left[\sup_{t \in[0,T]}|\Tilde{u}_{ik}(t)-\Tilde{v}_{ik}(t)|^2\right]
&\leq
C\frac{1}{MN}\sum_{j=1}^{\ngrid}\sum_{m=1}^\ndelta \int_0^T\E\left[|u_{jm}(r)-v_{jm}(r)|^2\right] \,dr
\\
& \leq
T C \frac{1}{MN} \sum_{j=1}^{\ngrid}\sum_{m=1}^\ndelta \E\left[\sup_{t\in[0,T]}|u_{jm}(t)-v_{jm}(t)|^2\right].
\end{aligned}
\end{gather*}
Finally we sum over $i=1,\dots,\ngrid$ and $k=1,\dots,\ndelta$. In conclusion, by taking another time $T^*<T$ small enough with respect to $C=C(T,b,\sigma)$, we find that the map $F:(H_{T^*}^2)^{\ngrid\ndelta}\to(H_{T^*}^2)^{\ngrid\ndelta}$ is indeed a contraction.
The unique fixed point $u_{ik}=F(u_{ik})\in H_{T^*}^2$ is then the (pathwise unique) solution on $[0,T^*]$.
We conclude by gluing solutions on subsequent intervals $[nT^*,(n+1)T^*]$ up to $[0,\infty)$.
\end{proof}


\section{Well-posedness of the limiting McKean--Vlasov SDEs and PDE}
\label{the limiting model}

In this section we analyze the limiting model for the particle system \eqref{abstract 4 model particle system}, that is the McKean--Vlasov equation \eqref{abstract 4 model McKean--Vlasov} and the nonlinear Fokker--Planck equation \eqref{abstract 4 model nonlinear fokker--planck}.
In particular, Theorems \ref{Strong existence and uniqueness for the McKean--Vlasov} and \ref{Well-posedness of the non-linear fokker--planck equations} about existence and uniqueness for these equations will be proved using a contraction argument.

Let us define the functional setting for the contraction argument.
The Banach space $H_T^2\coloneqq L^2(\Omega$ is defined as in \eqref{continuous second moment processes}.
For any $T>0$ we shall also consider the complete metric space $C_T^2\coloneqq C([0,T];\pr_2(\R^4))$ of continuous functions with values in the complete metric space $(\pr_2(\R^4),\w_2(\R^4))$, where $\w_2$ is the Wasserstein distance \eqref{wasserstein distance}, endowed with the supremum distance $d_{C_T^2}(f,g)=\sup_{t\in[0,T]}\w_2(f(t),g(t))$.
Finally, we will employ the Banach space $\linfproctwo$ of bounded measurable maps $Q\to H_T^2$ endowed with the norm 
\begin{equation*}
    |u|_{\linfproctwo}\coloneqq\sup_{x\in Q}|u(x,\cdot)|_{H_T^2}=\sup_{x\in Q}\E\left[\sup_{t\in[0,T]}|u(x,t)|^2\right]^{\frac{1}{2}}.
\end{equation*}
Similarly, we also make use of the space $\linfmeastwo$.
Notice that despite $C_T^2=C([0,T];\pr_2(\R^4))$ not being a vector space, it still makes sense to say that a function $f:Q\to C_T^2$ is bounded by taking an arbitrary point $P_0\in C_T^2$ and imposing $\sup_{x\in Q}d_{C_T^2}(f(x),P_0)<\infty$.
For simplicity, we take $P_0(t)\equiv\delta_0$ the function (in $t$) identically equal to $\delta_0\in\pr_2(\R^4)$ --- the Dirac mass centered at zero. With abuse of notation, we denote 
\begin{equation*}
    |f|_{\linfmeastwo}\coloneqq\sup_{x\in Q}d_{C_T^2}(f(x),\delta_0)=\sup_{x\in Q}\sup_{t\in[0,T]}\Big(\int_{\R^4}|v|^2\,f(t,x,dv)\,\Big)^{\frac{1}{2}}.
\end{equation*}
Then $\linfmeastwo$ is a complete metric space with the distance $d_{\linfmeastwo}(f,g)\coloneqq\sup_{x\in Q}d_{C_T^2}(f(x),g(x))$.

Let us now introduce the maps yielding the contraction.
We are interested in the composition
\begin{equation}
    \label{contraction map 2nd moments}
    \linfmeastwo\xrightarrow{S^{\epsilon}}\linfproctwo\xrightarrow{L}\linfmeastwo.
\end{equation}
The map $L$ sends an element $u\in\linfproctwo$ to its bounded-in-space and continuous-in-time law on $\R^4$.
That is to say $L[u](x,\cdot)\in C([0,T];\pr_2(\R^4))$ is given by $L[u](x,t)=\law_{\R^4}(u(x,t))$ for each $x\in Q$ and $t\in [0,T]$.  
A direct computation gives
\begin{gather}\label{map L is well-defined}
\begin{aligned}
    \sup_{x\in Q}d_{C_T^2}(L[u](x),\delta_0)
    & =
    \sup_{x\in Q}\sup_{t\in[0,T]}\w_2(L[u](x,t),\delta_0)
    \\
    & =
    \sup_{x\in Q}\sup_{t\in[0,T]}\E\left[|u(x,t)|^2\right]^{\frac{1}{2}}
     \leq
    \sup_{x\in Q}\E\left[\sup_{t\in[0,T]}|u(x,t)|^2\right]^{\frac{1}{2}}
    =
    |u|_{\linfproctwo}
    <\infty,
\end{aligned}
\end{gather}
and $L[u]$ is indeed an element of $\linfmeastwo$.

The map $S^{\epsilon}$ is defined by sending an element $f\in\linfmeastwo$ to the solutions $(S^{\epsilon}[f](x,t))_{t\geq0}$ of the following SDEs with reflecting boundary conditions: for each fixed $x\in Q$
\begin{align}[left ={\empheqlbrace}]\label{solution map S}
\small
    \begin{split}
    S^{\epsilon}[f](x,t)=&\,u(x,0)\!+\!\!\int_0^t\!\!b(x,r,S^{\epsilon}[f](x,r),f(r))\,dr+\int_0^t\!\!\sigma(x,r,S^{\epsilon}[f](x,r),f(r))\,dW^{\epsilon}(x,t)-\ell[f](x,t),
    \\
    \ell^{\beta}[f](x,t)=&\,-|\ell^{\beta}[f](x,\cdot)|(t),\, |\ell^{\beta}[f](x,\cdot)|(t)\!=\!\int_0^t\! \! 1_{\{S^{\epsilon}[f]^\beta(x,r)=0\}}d|\ell^{\beta}[f](x,\cdot)|(r),\,\,\beta = 1,2,3,4,
    \end{split}
\end{align}where $u(\cdot,0)\in L^{\infty}(Q;L^2(\Omega)$ is the initial condition for the McKean--Vlasov equation.
Notice that we slightly abuse notation since we identify an element $f\in\linfmeastwo$ with the time dependent probability measure $f(t,dx,du)$ on $Q\times\R^4$ defined by
\begin{equation}
    \int_{Q\times\R^4}\varphi(x,u) f(t,dx,du)\coloneqq\int_Q\int_{\R^4}\varphi(x,u)f(t,x,du)\,dx\qquad\text{for any} \quad \varphi\in C_b(Q\times\R^4).
\end{equation}
Standard theory of SDEs with reflecting boundary conditions \cite{Sznitman-1984-nonlinear-reflecting-diffusion} ensures that, for each fixed $x\in Q$, equation \eqref{solution map S} has a pathwise unique solution. Indeed, owing to the conditions \eqref{drift term structure}--\eqref{b_1 sigma_1 sublinear property} on $b$ and $\sigma$, for fixed $x\in Q$ and $f\in\linfmeastwo$ the drift $\Tilde{b}(r,u)\coloneqq b(x,r,u,f(r))$ and diffusion $\Tilde{\sigma}(r,u)\coloneqq \sigma(x,r,u,f(r))$ terms can be verified to satisfy the needed assumptions.
The measurability of $x\mapsto S^{\epsilon}[f](x,\cdot)\in H^2_T$ then immediately follows from that of the initial data $u(x,0)$ and of the noise $W^{\epsilon}(x,t)$, using a Picard iteration representation of the solution of the SDEs \eqref{solution map S}.
The fact that the map $S$ is well-defined, i.e. that $S^{\epsilon}[f](x,\cdot)$ is indeed an element of $H_T^2$ uniformly bounded in $x\in Q$, is the subject of Lemma \ref{a priori estimates 2nd moments}.

By definition, for every $f\in\linfmeastwo$ we have $(L\circ S^{\epsilon})[f](x,t)=\law_{\R^4}(S^{\epsilon}[f](x,t))$ for all $x\in Q$ and $t\in[0,T]$, but in fact we can say that $(L\circ S^{\epsilon})[f](x,\cdot)=\law_{C([0,T];\R^4)}(S^{\epsilon}[f](x,\cdot))$, since it is the law of the SDE \eqref{solution map S}. That is, $(L\circ S^{\epsilon})[f](x,\cdot)$ can be seen as a probability measure on the space of continuous paths $C([0,T];\R^4)$. Furthermore, if $f\in\linfmeastwo$ is a fixed point of $L\circ S^{\epsilon}$, namely $f(x,t)=\law_{\R^4}(S^{\epsilon}[f](x,t))$ for all $x\in Q$ and $t\in[0,T]$, then $f(x,\cdot)=\law_{C([0,T];\R^4)}(S^{\epsilon}[f](x,\cdot))$.

\begin{lemma}[A priori estimates on moments]\label{a priori estimates 2nd moments}
Given $f\in\linfmeastwo$, the pathwise unique solution $(S^{\epsilon}[f](x,t))_{x\in Q,\,t\geq0}$ to \eqref{solution map S} satisfies
\begin{align}
    \label{a priori estimates 2nd moments eq 1}
    |S^{\epsilon}[f]|_{\linfproctwo}^2
    &=\sup_{x\in Q}\E\left[\sup_{t\in[0,T]}|S^{\epsilon}[f](x,t)|^2\right]\nonumber
    \\
    &\leq
    C\left(1+\sup_{x\in Q}\E\left[|u(x,0)|^2\right]+\int_0^T\sup_{x\in Q}\int_{\R^4}|u|^2\,f(t,x,du)\,dt\right) 
    \\
    &\leq
    C\left(1+\sup_{x\in Q}\E\left[|u(x,0)|^2\right]+|f|_{\linfmeastwo}^2\right),\nonumber
\end{align}
for a constant $C=C(T,b,\sigma)$. In particular, $S^{\epsilon}[f]\in\linfproctwo$, and the map $S^{\epsilon}$ and the composition $L\circ S^{\epsilon}$ are well defined.
Moreover, if $f\in\linfmeastwo$ is a fixed point of $L\circ S^{\epsilon}$, then $f(x,\cdot)$, as a probability measure on the space of continuous paths $C([0,T];\R^4)$, satisfies the stronger bound
\begin{equation}
    \label{a priori estimates 2nd moments eq 2}
    \sup_{x\in Q}\int_{C([0,T];\R^4)}\sup_{t\in[0,T]}|v(t)|^2\,f(x,dv)=|S^{\epsilon}[f]|_{\linfproctwo}^2
    \leq
    C\Big(1+\sup_{x\in Q}\E\left[|u(x,0)|^2\right]\Big).
\end{equation}
\end{lemma}
\begin{proof}
Fix any $f\in\linfmeastwo$, we want to estimate $|S^\epsilon[f](x,t)|^2$.
Owing to the structure \eqref{drift term structure}--\eqref{diffusion term structure} and the sublinear growth properties \eqref{b_0 sigma_0 sublinear property}--\eqref{b_1 sigma_1 sublinear property} of the drift and diffusion terms, we have
\begin{equation}\label{a priori estimates 2nd moments eq 3}
    |b(x,r,u,f(r))|+|\sigma(x,r,u,f(r))|\leq C \left(1+|u|+\sup_{y\in Q}\int_{\R^4}|v|\,f(r,y,dv)\,\right).
\end{equation}
Then, using H\"older's inequality one gets
\begin{equation}
    \label{a priori estimates 2nd moments eq 4}
    \sup_{y\in Q}\!\int_{\R^4}\!\!|v|\,f(r,y,dv)
    \leq
    \sup_{y\in Q}\Big(\!\int_{\R^4}|v|^2\,f(r,y,dv)\Big)^{\frac{1}{2}}
    \!\!\!\leq
    \sup_{y\in Q}\sup_{r\in[0,T]}\!\!\Big(\!\int_{\R^4}|v|^2\,f(r,y,dv)\Big)^{\frac{1}{2}}\!\!=|f|_{\linfmeastwo}.
\end{equation}
Moreover, the explicit details in \cite{Sznitman-1984-nonlinear-reflecting-diffusion} on the construction of the reflection term $\ell$ in the SDE \eqref{solution map S} imply that we can control it as follows:
\begin{equation}
    \label{a priori estimates 2nd moments eq 5}
    |\ell[f](x,t)|
    \!\leq
    \!\sup_{\tau\in[0,t]}\!\left|u(x,0)\!+\!\int_0^\tau\!\!\! b(x,r,S^{\epsilon}[f](x,r),f(r))\,dr\!+\!\int_0^\tau \!\!\!\sigma(x,r,S^{\epsilon}[f](x,r),f(r))\,dW^{\epsilon}(x,r)\right|.
\end{equation}
Squaring both sides of the SDE \eqref{solution map S}, controlling the reflection term with the estimate \eqref{a priori estimates 2nd moments eq 5}, applying convexity inequalities, taking the supremum over $t\in[0,T]$ and then the expectation, and finally handling the deterministic integral with H\"older's inequality and the stochastic integral with It\^o isometry, we obtain
\begin{gather}\label{a priori estimates 2nd moments eq 5.5}
\begin{aligned}
    \E\left[\sup_{t\in[0,T]}|S^{\epsilon}[f](x,t)|^2\right]
    \leq
    C\bigg(&\E\left[u(x,0)^2\right]+\E\left[\int_0^T b(x,r,S^{\epsilon}[f](x,r),f(r))^2\,dr\right]\,dt
    \\
    &+\E\left[\int_0^T \sigma(x,r,S^{\epsilon}[f](x,r),f(r))^2\,dr)\right]\bigg),
\end{aligned}
\end{gather}
for a numeric constant $C$.
In turn, using the sublinear growth estimates \eqref{a priori estimates 2nd moments eq 3}--\eqref{a priori estimates 2nd moments eq 4}, we get
\begin{multline}
\label{a priori estimates 2nd moments eq 6}
    \E\left[\sup_{t\in[0,T]}|S^\epsilon[f](x,t)|^2\right]
    \leq
    C\Bigg(1+\E\left[u(x,0)^2\right] 
    \\
    \left.+\int_0^T\E\left[\sup_{r\in[0,t]}|S^\epsilon[f](x,r)|^2\right]\,dt
    +\int_0^T\sup_{y\in Q}\int_{\R^4}|u|^2\,f(t,y,du)\,dt\right),
\end{multline}
for a constant $C=C(T,b,\sigma)$.
Then we exploit Gr\"onwall's Lemma to get rid of the third term on the right hand side of the inequality in \eqref{a priori estimates 2nd moments eq 6}. Eventually, by taking the supremum over $x\in Q$ and using \eqref{a priori estimates 2nd moments eq 4} again, we deduce the inequalities \eqref{a priori estimates 2nd moments eq 1}.

Suppose now that $f$ is a fixed point of $L\circ S^{\epsilon}$.
Since $f(t,y)=\law_{\R^4}(S^{\epsilon}[f](y,t))$, we readily verify that
\begin{equation*}
    \int_0^T\sup_{y\in Q}\int_{\R^4}|u|^2\,f(t,y,du)\,dt
    =
    \int_0^T\sup_{y\in Q}\E\left[|S^{\epsilon}[f](y,t)|^2\right]\,dt
    \leq
    \int_0^T\sup_{y\in Q}\E\left[\sup_{r\in[0,t]}|S^{\epsilon}[f](y,r)|^2\right]\,dt.
\end{equation*}
Then we exploit this bound and again use Gr\"onwall's Lemma in the first inequality \eqref{a priori estimates 2nd moments eq 1} to get rid of the third term at the right hand side and obtain \eqref{a priori estimates 2nd moments eq 2} in the statement.
\end{proof}

Now, assuming that the composition $L\circ S^{\epsilon}$ is a contraction in $\linfmeastwo$, we first show how to conclude the strong existence and uniqueness for the McKean--Vlasov equation \eqref{abstract 4 model McKean--Vlasov}.
Let $f\in\linfmeastwo$ be the unique fixed point of $L\circ S^{\epsilon}$: since $S^{\epsilon}[f]$ solves \eqref{solution map S} and $L\circ S^{\epsilon}[f]=f$, we obtain that $S^{\epsilon}[f]$ solves the McKean--Vlasov equation \eqref{abstract 4 model McKean--Vlasov} on our stochastic basis with initial data $(u(x,0))_{x\in Q}$.
Conversely, let $(u(x,t))_{x\in Q,\,t\geq0}$ be a strong solution of \eqref{abstract 4 model McKean--Vlasov} on our stochastic basis with these initial data, then $L[u]\in\linfmeastwo$ is a fixed point of $L\circ S^{\epsilon}$ and thus we must have $L[u]=f$, the unique fixed point; but then, since we have strong uniqueness for the SDEs \eqref{solution map S} defining the map $S^{\epsilon}$ and since $u(x,t)$ solves these SDEs with this data $f$, we conclude that $(u(x,t))_{x\in Q,\,t\geq 0}=S^{\epsilon}[f]$.

\begin{proof}[\textbf{Proof of Theorem \ref{Strong existence and uniqueness for the McKean--Vlasov}}]
\noindent
 We show that the mapping $L\circ S^\epsilon$ is a strict contraction and then apply the Banach fixed point theorem.
Take $f,g\in\linfmeastwo$.
By definition of the map $S^\epsilon$ we have
\begin{gather}
    S^{\epsilon}[f](x,t)=u(x,0)+\int_0^tb(x,r,S^{\epsilon}[f],f)\,dr+\int_0^t\sigma(x,r,S^{\epsilon}[f],f)\,dW^{\epsilon}(x,r)-\ell[f](x,t), \label{strong existence uniqueness McKean--Vlasov proof eq 1}
    \\
    S^{\epsilon}[g](x,t)=u(x,0)+\int_0^tb(x,r,S^{\epsilon}[g],g)\,dr+\int_0^t\sigma(x,r,S^{\epsilon}[g],g)\,dW^{\epsilon}(x,r)-\ell[g](x,t).\label{strong existence uniqueness McKean--Vlasov proof eq 2}
\end{gather}
Then, by definition of the map $L$, we have that $(S^\epsilon[f](x,t),S^\epsilon[g](x,t))$ is an admissible coupling for $(L\circ S^\epsilon[f](x,t),L\circ S^\epsilon[g](x,t))$ and we can use it to estimate $\w_2(L\circ S^\epsilon[f](x,t),L\circ S^\epsilon[g](x,t))$.

We take the difference of equations \eqref{strong existence uniqueness McKean--Vlasov proof eq 1} and \eqref{strong existence uniqueness McKean--Vlasov proof eq 2} and use the It\^o formula to get
\begin{gather}\label{strong existence uniqueness McKean--Vlasov proof eq 3}
\begin{aligned}
    |S^\epsilon[f]&(x,t)-S^\epsilon[g](x,t)|^2
    \\
    =&\,2\int_0^t\!\!\!\big(S^\epsilon[f](x,r)-S^\epsilon[g](x,r)\big)\big(b(x,r,S^\epsilon[f],f)-b(x,r,S^\epsilon[g],g)\big)\,dr 
    \\
    &+
    2\!\int_0^t\!\!\!\big(S^\epsilon[f](x,r)\!-\!S^\epsilon[g](x,r)\big)\big(\sigma(x,r,S^\epsilon[f],f)\!-\!\sigma(x,r,S^\epsilon[g],g)\big)\,dW^\epsilon(x,r)
    \\
    &+
    2\int_0^t\big(S^\epsilon[f](x,r)-S^\epsilon[g](x,r)\big)\big(d\ell[g](x,r)-d\ell[f](x,r)\big)
    \\
    &+
    \int_0^t\!\!\!\big(\sigma(x,r,S^\epsilon[f](x,r),f(r))-\sigma(x,r,S^\epsilon[g](x,r),g(r))\big)^2\,dr.
\end{aligned}
\end{gather}
We now argue analogously to \eqref{strong existence and uniqueness 2nd moments proof 2}--\eqref{strong existence and uniqueness 2nd moments proof 5} in the proof of Theorem \ref{Strong existence and uniqueness for the particle systems}.
First, as in \eqref{strong existence and uniqueness 2nd moments proof 3} the third 
the second term on the right hand side of \eqref{strong existence uniqueness McKean--Vlasov proof eq 3} is negative, and we drop it. 
Then we take the supremum in time and apply the expectation, we control the first deterministic integral with H\"older's inequality and the stochastic integral with the Burkholder--Davis--Gundy and H\"older's inequality, and finally we absorb the necessary terms on the left hand side of \eqref{strong existence uniqueness McKean--Vlasov proof eq 3} to get
\begin{gather}\label{strong existence uniqueness McKean--Vlasov proof eq 4}
\begin{aligned}
    \E\left[\sup_{r\in[0,t]}|S^\epsilon[f](x,r)-S^\epsilon[g](x,r)|^2\right]
    \leq
    C\,\E\bigg[&\int_0^t\!\!\!\big(b(x,r,S^\epsilon[f],f)-b(x,r,S^\epsilon[g],g)\big)^2\,dr
    \\
    &+\int_0^t\big(\sigma(x,r,S^\epsilon[f],f)-\sigma(x,r,S^\epsilon[g],g)\big)^2\,dr\bigg],
\end{aligned}
\end{gather}
for a numeric constant $C$.
Now we exploit the Lipschitz properties of the drift and diffusion terms stated in Lemma \ref{b sigma lipschitz and sublinear properties if disintegration} and we obtain, for $C=C(b,\sigma)$,
\begin{gather}\label{strong existence uniqueness McKean--Vlasov proof eq 5}
\begin{aligned}
\E\left[\sup_{r\in[0,t]}|S^\epsilon[f](x,r)-S^\epsilon[g](x,r)|^2\right]
    \leq
    C\bigg(&\int_0^t\E\left[|S^\epsilon[f](x,r)-S^\epsilon[g](x,r)|^2\right]\,dr
    \\
    &+
    \int_0^t\sup_{y\in Q}\w_2(f(r,y,du),g(r,y,du))^2\,dr\bigg).
\end{aligned}
\end{gather}
Using Gr\"onwall's Lemma we get rid of the first term on the right hand side at the expense of a larger constant $C=C(T,b,\sigma)$.
Moreover, we have $\w_2(f(r,y,du),g(r,y,du))\leq\sup_{r\in[0,T]}\w_2(f(r,y,du),g(r,y,du))$ for any $r\in[0,T]$, and we conclude that
\begin{equation}\label{strong existence uniqueness McKean--Vlasov proof eq 6}
    \E\left[\sup_{t\in[0,T]}|S^\epsilon[f](x,t)-S^\epsilon[g](x,t)|^2\right]\leq C\,T\,\sup_{y\in Q}\sup_{r\in[0,T]}\w_2(f(r,y,du),g(r,y,du))^2,
\end{equation}
for a constant $C=C(T,b,\sigma)$.
Finally, since the right hand side is independent of $x$, we take the supremum over $x\in Q$ on the left hand side of \eqref{strong existence uniqueness McKean--Vlasov proof eq 6}.

In conclusion, recalling that $(S^\epsilon[f](x,t),S^\epsilon[g](x,t))$ is a coupling for $(L\circ S^\epsilon[f](x,t),L\circ S^\epsilon[g](x,t))$, we obtain
\begin{align*}\label{strong existence uniqueness McKean--Vlasov proof eq 7}
    d_{\linfmeastwo}(L\circ S^\epsilon[f],L\circ S^\epsilon[g])^2
    &=
    \sup_{x\in Q}\sup_{t\in[0,T]}\w_2(L\circ S^\epsilon[f](x,t),L\circ S^\epsilon[g](x,t))^2
    \\
    &\leq
    \sup_{x\in Q}\sup_{t\in[0,T]}\E\left[|S^\epsilon[f](x,t)-S^\epsilon[g](x,t)|^2\right]
    \\
    &
    \leq
    \sup_{x\in Q}\E\left[\sup_{t\in[0,T]}|S^\epsilon[f](x,t)-S^\epsilon[g](x,t)|^2\right]
    \\
    &
    \leq
    C\,\,T\,\,\sup_{x\in Q}\sup_{t\in[0,T]}\w_2(f(x,t),g(x,t))^2
    =
    C\,\,T\,\,d_{\linfmeastwo}(f,g)^2,
\end{align*}
for $C=C(T,b,\sigma)$.
This constant $C$ is increasing in $T$. Therefore, upon possibly working in $L^{\infty}(Q;C_{T^*}^2)$ for some smaller $T^*<T$, we can assume that $CT<1$.
That is to say, if $T>0$ is small enough, we have a contraction in $\linfmeastwo$.
In turn this implies that we have a pathwise unique solution to \eqref{abstract 4 model McKean--Vlasov} over $[0,T]$.
Repeating the same argument over $[T,2T]$, $[2T,3T]$ and so on, and exploiting the uniqueness, we can show there exists a pathwise unique solution defined over all $[0,\infty)$.

Now, assume in addition that $u(x,0)\in C^{\alpha}(Q;L^2(\Omega))$. The following argument proves that in this case, for any $f\in\linfmeastwo$, we have $S^{\epsilon}[f]\in C^{\alpha}(Q;L^2(\Omega;C([0,T];\R^4)))$.
In particular, the solution of the McKean--Vlasov equation \eqref{abstract 4 model McKean--Vlasov} satisfies $u^\epsilon(x,t)\in C^{\alpha}(Q;L^2(\Omega;C([0,T];\R^4)))$.

Given $x,y\in Q$, we manipulate the equations \eqref{solution map S} for $S^{\epsilon}[f](x)$ and $S^{\epsilon}[f](y)$ to write 
\begin{gather}\label{strong existence uniqueness McKean--Vlasov proof eq 7.bis}
    \begin{aligned}
    S^\epsilon[f](x,t)-S^\epsilon[f](y,t)
    =&\,\big(u(x,0)-u(y,0)\big)
    \\
    &+\int_0^t\!\!\big(b(x,r,S^\epsilon[f](x,r),f(r))-b(x,r,S^\epsilon[f](y,r),f(r))\big)\,dr 
    \\
    &+
    \int_0^t\!\!\big(\sigma(x,r,S^\epsilon[f](x,r),f(r))\!-\!\sigma(y,r,S^\epsilon[f](y,r),f(r))\big)\,dW^{\epsilon}(x,r)
    \\
    &+\int_0^t\!\!\sigma(y,r,S^\epsilon[f](y,r),f(r))
    \left(dW^\epsilon(x,r)-dW^{\epsilon}(y,r)\right) 
    \\
    &+\big(\ell[f](y,r)-\ell[f](x,r)\big).
    \end{aligned}
\end{gather}
Applying the It\^o formula to the squared power yields
\begin{gather}\label{strong existence uniqueness McKean--Vlasov proof eq 8}
    \begin{aligned}
    |S^\epsilon[f](x,t)-&S^\epsilon[f](y,t)|^2
    \\
    =&\,(u(x,0)-u(y,0))^2
    \\
    &+\,2\int_0^t\!\!\!\big(S^\epsilon[f](x,r)-S^\epsilon[f](y,r)\big)\big(b(x,r,S^\epsilon[f],f)-b(x,r,S^\epsilon[f],f)\big)\,dr 
    \\
    &+
    2\!\int_0^t\!\!\!\big(S^\epsilon[f](x,r)\!-\!S^\epsilon[f](y,r)\big)\big(\sigma(x,r,S^\epsilon[f],f)\!-\!\sigma(y,r,S^\epsilon[f],f)\big)\,dW^\epsilon(x,r)
    \\
    &+
    2\!\int_0^t\!\!\!\big(S^\epsilon[f](x,r)\!-\!S^\epsilon[f](y,r)\big)\,\sigma(y,r,S^\epsilon[f],f)\left(dW^\epsilon(x,r)-dW^{\epsilon}(y,r)\right) 
    \\
    &+
    2\int_0^t\big(S^\epsilon[f](x,r)-S^\epsilon[f](y,r)\big)\big(d\ell[f](y,r)-d\ell[f](x,r)\big)
    \\
    &+
    \int_0^t\!\!\!\big(\sigma(x,r,S^\epsilon[f](x,r),f(r))-\sigma(y,r,S^\epsilon[f](y,r),f(r))\big)^2\,dr
    \\
    &+
    \int_0^t\!\sigma(y,r,S^\epsilon[f](y,r),f(r))^2\,d\left[W^{\epsilon}(x,r)-W^\epsilon(y,r)\right].
    \end{aligned}
\end{gather}
For the first stochastic integral, the Burkholder--Davis--Gundy inequality and H\"older's inequality yield
\begin{gather}\label{strong existence uniqueness McKean--Vlasov proof eq 9}
    \begin{aligned}
    \E\Bigg[&\sup_{t\in[0,T]}\left|\int_0^t\!\!\!\big(S^\epsilon[f](x,r)\!-\!S^\epsilon[f](y,r)\big)\big(\sigma(x,r,S^\epsilon[f],f)\!-\!\sigma(y,r,S^\epsilon[f],f)\big)\,dW^\epsilon(x,r)\right|\Bigg]
    \\
    &\leq
    \E\!\left[\!\left(\int_0^T\!\!\!\!\!\big(S^\epsilon[f](x,r)\!-\!S^\epsilon[f](y,r)\big)^2\!\big(\sigma(x,r,S^\epsilon[f],f)\!-\!\sigma(y,r,S^\epsilon[f],f\big)^2dr\right)^{\frac{1}{2}}\!\right]
    \\
    &\leq
    \E\left[\!\sup_{t\in[0,T]}\!\big|S^\epsilon[f](x,t)\!-\!S^\epsilon[f](y,t)\big|\left(\int_0^T\!\!\!\!\!\big(\sigma(x,r,S^\epsilon[f],f)\!-\!\sigma(y,r,S^\epsilon[f],f\big)^2dr\right)^{\frac{1}{2}}\!\!\right]
    \\
    &\leq
    \delta\,\E\left[\sup_{t\in[0,T]}\big|S^\epsilon[f](x,t)\!-\!S^\epsilon[f](y,t)\big|^2\right]
    +
    \frac{1}{\delta}\,\E\left[\int_0^T\!\!\!\big(\sigma(x,r,S^\epsilon[f],f)\!-\!\sigma(y,r,S^\epsilon[f],f\big)^2dr\right],
    \end{aligned}
\end{gather}
where $\delta>0$ shall be chosen small enough so as to absorb the first term on the right hand side.
Similarly, for the second stochastic integral we find
\begin{gather}\label{strong existence uniqueness McKean--Vlasov proof eq 9.bis}
    \begin{aligned}
    \E&\Bigg[\sup_{t\in[0,T]}\left|\int_0^t\!\!\!\big(S^\epsilon[f](x,r)\!-\!S^\epsilon[f](y,r)\big)\,\sigma(y,r,S^\epsilon[f](y,r),f(r))\left(dW^\epsilon(x,r)-dW^\epsilon(y,r\right)\right|\Bigg]
    \\
    &\leq
    \E\!\left[\!\left(\int_0^T\!\!\!\!\!\big(S^\epsilon[f](x,r)\!-\!S^\epsilon[f](y,r)\big)^2\,\sigma(y,r,S^\epsilon[f],f)^2\,d\left[W^{\epsilon}(x,r)-W^\epsilon(y,r)\right] \right)^{\frac{1}{2}}\!\right]
    \\
    &\leq
    \delta\,\E\left[\sup_{t\in[0,T]}\big|S^\epsilon[f](x,t)\!-\!S^\epsilon[f](y,t)\big|^2\right]
    +
    \frac{1}{\delta}\,\E\left[\int_0^T\!\!\!\sigma(y,r,S^\epsilon[f],f)^2d\left[W^{\epsilon}(x,r)-W^\epsilon(y,r)\right]\right],
    \end{aligned}
\end{gather}
where again $\delta>0$ shall be chosen small enough to absorb the first term on the right hand side.

We now go back to \eqref{strong existence uniqueness McKean--Vlasov proof eq 8}.
As in \eqref{strong existence and uniqueness 2nd moments proof 2}, the third term on the right hand side is always negative and we drop it.
Then we take the supremum in time and we apply the expectation, we handle the first deterministic integral with H\"older's inequality and we use estimates \eqref{strong existence uniqueness McKean--Vlasov proof eq 9} and \eqref{strong existence uniqueness McKean--Vlasov proof eq 9.bis} for the stochastic integrals, absorbing the necessary terms on the left hand side by choosing $\delta$ small enough.
We obtain, for a numeric constant $C$,
\begin{gather}\label{strong existence uniqueness McKean--Vlasov proof eq 10}
    \begin{aligned}
    \E\Bigg[\!\sup_{t\in[0,T]}\!|S^\epsilon[f](x,t)-S^\epsilon[f](y,t)|^2\Bigg]
    \leq C\Bigg(&\E\left[|u(x,0)-u(y,0)|^2\right]
    \\
    &+\!\int_0^T\E\left[\left|S^\epsilon[f](x,r)-S^\epsilon[f](y,r)\right|^2\right]dr
    \\
    &+\!
    \int_0^T\E\left[\left|b(x,r,S^\epsilon[f],f)\!-\!b(y,r,S^\epsilon[f],f\big)\right|^2\right]dr
    \\
    &+\!
    \int_0^T\E\left[\left|\sigma(x,r,S^\epsilon[f],f)\!-\!\sigma(y,r,S^\epsilon[f],f)\right|^2\right]dr
    \\
    &+\!
    \E\left[\int_0^T\!\!\!\sigma(y,r,S^\epsilon[f](y,r),f(r))^2\,d\left[W^{\epsilon}(x,r)\!-\!W^\epsilon(y,r)\right]\right]
    \Bigg).
    \end{aligned}
\end{gather}
Now we recall formula \eqref{formula/epsilon convoluted noise quadratic variation} for the quadratic variation of $W^{\epsilon}(x,t)-W^\epsilon(y,t)$, we use the Lipschitz and H\"older properties \eqref{lipschitz and holder properties for f fixed} of $b$ and $\sigma$ and convexity inequalities to get
\begin{gather}\label{strong existence uniqueness McKean--Vlasov proof eq 11}
    \begin{aligned}
    \E\Bigg[\sup_{t\in[0,T]}|S^\epsilon&[f](x,t)-S^\epsilon[f](y,t)|^2\Bigg]
    \\
    \leq C\Bigg(&\E\left[|u(x,0)-u(y,0)|^2\right]
    +\int_0^T\E\left[\left|S^\epsilon[f](x,r)-S^\epsilon[f](y,r)\right|^2\right]dr+|x-y|^{2\alpha}
    \\
    &
    +\frac{|x-y|^2}{\epsilon^2}\int_0^T\!\E\left[|\sigma(y,r,S^{\epsilon}[f],f)|^2\right]dr
    \Bigg),
    \end{aligned}
\end{gather}
for a constant $C=C(T,b,\sigma,\rho)$.
We get rid of the second term on the right hand side of \eqref{strong existence uniqueness McKean--Vlasov proof eq 11}  with Gr\"onwall's Lemma, at the price of a larger constant $C=C(T,b,\sigma)$.
The first term is handled with the assumption $u(\cdot,0)\in C^{\alpha}(Q;L^2(\Omega)$.
We control the last term with the sublinear growth property \eqref{lipschitz and holder properties for f fixed} of $\sigma$ and the a priori estimate \eqref{a priori estimates 2nd moments eq 1}.
In conclusion we obtain
\begin{gather}\label{strong existence uniqueness McKean--Vlasov proof eq 12}
\small
    \begin{aligned}
    \E\left[\!\sup_{t\in[0,T]}\!|S^\epsilon[f](x,t)\!-\!S^\epsilon[f](y,t)|^2\right]
    \!\leq\! C\left(\!|x-y|^{2\alpha}\!+\!\frac{|x-y|^2}{\epsilon^2}\bigg(1+\sup_{z\in Q}\E\left[|u(z,0)|^2\right]+|f|_{\linfmeastwo}^2\bigg)\right),
    \end{aligned}
\end{gather}
for a constant $C=C(T,b,\sigma,\rho,[u(\cdot,0)]_{\alpha})$.
Since $x,y\in Q$ are arbitrary, this concludes the proof that $S^\epsilon[f]\in C^\alpha(Q;L^2(\Omega;C([0,T];\R^4)))$.
\end{proof}

We end this section by proving the existence and uniqueness of solutions to the associated Fokker--Planck equation.

\begin{proof}[\textbf{Proof of Theorem \ref{Well-posedness of the non-linear fokker--planck equations}}]
\noindent
The result is a consequence of the It\^o formula, the same fixed point argument as for the McKean--Vlasov equation and the uniqueness statement for the \emph{linear} version of the Fokker--Planck type equation.
Given any admissible initial condition $f_0(x,du)\in L^{\infty}(Q;\pr_2(\R^4))$, standard probability theory ensures that we can find a probability space $(\Omega, \F, \p)$ supporting a $4$-dimensional space-time white noise $(W(x,t))_{x\in Q,t\geq0}$ and a family of random variables $u(x,0)\in L^{\infty}(Q;L^2(\Omega))$ independent of the noise $W(x,t)$ with $\law_{\R^4}(u(x,0))=f_0(x,du)$ for every $x\in Q$.
Given any $\epsilon>0$, we convolve and rescale the white noise to obtain $W^{\epsilon}(x,t)$ as in \eqref{formula/epsilon convoluted and rescaled noise}.
With this stochastic basis and initial data, let $(\Bar{u}^{\epsilon}(x,t)\in L^{\infty}(Q;H_T^2)$ be the solution of the $\epsilon$-correlated McKean--Vlasov equation \eqref{abstract 4 model McKean--Vlasov}, whose existence is guaranteed by Theorem \ref{Strong existence and uniqueness for the McKean--Vlasov}, and let us denote $f^\epsilon(x,t,du)=\law_{\R^4}(\Bar{u}^\epsilon(x,t))\in \linfmeastwo$.
We claim that $f^\epsilon$ is a weak solution of equation \eqref{abstract 4 model nonlinear fokker--planck}.

Take any $\phi\in C_c^2(\R_+\times \R^4)$ satisfying the Neumann boundary condition
\begin{equation} \label{Well-posedness of the non-linear fokker--planck equations proof eq 1}
    \nabla_u\phi(t,u)\cdot n_{\partial(\R^4_+)}(u)=0\quad \text{for all $t,u\in\R^+\times \partial(\R_+^4)$},
\end{equation}
where $n_{\partial(\R^4_+)}(u)$ denotes the unit outward normal at $u$.
An application of the It\^o formula yields
\begin{align} \label{Well-posedness of the non-linear fokker--planck equations proof eq 2}
    \phi(t,\Bar{u}^\epsilon(x,t))
    =&\phi(0,\Bar{u}(x,0))
    \!+\!
    \int_0^t\!\!\partial_t\phi(r,\Bar{u}^\epsilon(x,r))\,dr \nonumber
    +
    \!\int_0^t\!\nabla_u\phi(r,\Bar{u}^\epsilon(x,r))\cdot b(x,r,\Bar{u}^\epsilon(x,r),f^\epsilon(r))\,dr
    \\
    &+\!\int_0^t\!\nabla_u\phi(r,\Bar{u}^\epsilon(x,r))\cdot \sigma(x,r,\Bar{u}^\epsilon(x,r),f^\epsilon(r))\,dW^\epsilon(x,r)
    \\
    &+
    \int_0^t\sum_{\beta=1}^4\partial_{u^\beta}\phi(r,\Bar{u}^\epsilon(x,r))\,1_{\{\Bar{u}^{\epsilon,\beta}(x,r)=0\}}\,d|\ell^{\beta}(x,\cdot)|(r) \nonumber
    \\
    &+
    \int_0^t\frac{1}{2}\sum_{\beta=1}^4\partial^2_{u^\beta u^\beta}\phi(r,\Bar{u}^\epsilon(x,r))\,\,\big(\sigma_{\beta}(x,r,\Bar{u}^\epsilon(x,r),f^\epsilon(r))\big)^2\,dr. \nonumber
\end{align}
The fifth term on the right hand side is identically zero thanks to the condition \eqref{Well-posedness of the non-linear fokker--planck equations proof eq 1} on $\phi$.
Now we apply the expectation on both sides.
The fourth term on the right hand side vanishes by the martingale property of the stochastic integral.
Recalling that $\Bar{u}^\epsilon(x,t)$ takes values in $\R_+^4$ only, we get
\begin{align} \label{Well-posedness of the non-linear fokker--planck equations proof eq 3}
    \int_{\R^4_+}\phi(t,u)f^\epsilon(t,x,du)
    = &
    \int_{\R^4_+}\phi(0,u)f_0(x,du)
    +
    \int_0^t\int_{\R^4_+}\partial_t\phi(r,u)f^\epsilon(r,x,du)\,dr \nonumber
    \\
    &+
    \int_0^t\int_{\R^4_+}\nabla_u\phi(r,u)\cdot b(x,r,u,f^\epsilon(r))f^\epsilon(r,x,du)\,dr
    \\
    &+
    \int_0^t\int_{\R_+^4}\frac{1}{2}\sum_{\beta=1}^4\partial^2_{u^\beta u^\beta}\phi(r,u)\,\,\big(\sigma_{\beta}(x,r,u,f^\epsilon(r))\big)^2\,f^\epsilon(r,x,du)\,dr. \nonumber
\end{align}
This is nothing but the weak formulation of \eqref{abstract 4 model nonlinear fokker--planck} subjected to the no-flux boundary conditions.
Since for every $T>0$ we have $\Bar{u}^\epsilon(x,\cdot)\in H_T^2$ and since it satisfies the bound \eqref{Strong existence and uniqueness for the McKean--Vlasov eq 1}, we conclude that $f^\epsilon(x,t)=\law_{\R^4}(\Bar{u}^\epsilon(x,t))$ is a weak solution of \eqref{abstract 4 model nonlinear fokker--planck} with initial condition $f_0(x,du)$, that it lies in the space $L^{\infty}(Q;C([0,\infty);\pr_2(\R^4)))$ and that it actually satisfies the stronger bound \eqref{Well-posedness of the non-linear fokker--planck equations eq 1}. 

Conversely, let $g\in L^{\infty}(Q;C([0,\infty);\pr_2(\R^4)))$ be a weak solution of the non-linear Fokker--Planck equation \eqref{abstract 4 model nonlinear fokker--planck} with the same initial data $f_0$. We claim that $g=f^\epsilon$.
First, we can solve the family of standard SDEs with reflecting boundary conditions for the chosen $g$, for $x\in Q$:
\begin{align*}[left ={\empheqlbrace}]\label{Well-posedness of the non-linear fokker--planck equations proof eq 4}
    \begin{split}
    \displaystyle
    \,\,\,S^\epsilon[g](x,t)=&\,u(x,0)+\int_0^tb(r,x,S^\epsilon[g](x,r),g(r))\,dr\!+\!\!\int_0^t\!\!\sigma(r,x,S^\epsilon[g](x,r),g(r))\,dW^\epsilon(x,r)-\ell[g](x,t),
    \\
    \,\,\,\ell^{\beta}[g](x,t)=&-|\ell^{\beta}[g](x,\cdot)|(t)\,,\,\, |\ell^{\beta}[g](x,\cdot)|(t)\!=\!\int_0^t1_{\{S^\epsilon[g]^\beta(x,r)=0\}}d|\ell^{\beta}[g](x,\cdot)|(r)\quad\text{for }\beta=1,2,3,4.
    \end{split}
\end{align*}
Arguing as in \eqref{Well-posedness of the non-linear fokker--planck equations proof eq 2}--\eqref{Well-posedness of the non-linear fokker--planck equations proof eq 3}, we see that $h\coloneqq L\circ S^\epsilon[g]$ now solves the \emph{linear} Fokker--Planck equation with this fixed $g$ and with the same initial data $f_0$:
\begin{align}[left ={\empheqlbrace}]\label{Well-posedness of the non-linear fokker--planck equations proof eq 5}
    \begin{split}
        & \partial_th(t,x,u)+\nabla_u\cdot\big(b(x,t,u,g(t))h(t,x,u)\big)=\frac{1}{2}\sum_{\beta=1}^4\partial^2_{u^\beta u^\beta}\big(\sigma_\beta(x,t,u,g(t))^2h(t,x,u)\big),
        \\
        & b^\beta(x,t,u,g(t))h(t,x,u)-\frac{1}{2}\frac{\partial}{\partial u^\beta}\big(\sigma_{\beta}(x,t,u,g(t))^2h(t,x,u)\big)\Big|_{u^\beta=0}=0\quad\text{for $\beta=1,2,3,4$}.
    \end{split}
\end{align}
This linear equation is readily verified to satisfy uniqueness by a duality argument: indeed, for fixed $x\in Q$, it suffices to test it against arbitrary functions $\varphi(t,u)$ satisfying the so-called \emph{backward Kolmogorov equation} with Neumann boundary conditions on $\R_+^4$.
That is,
\begin{align*}[left ={\empheqlbrace}]\label{Well-posedness of the non-linear fokker--planck equations proof eq 6}
    \begin{split}
        & \partial_t\varphi+\nabla_u\varphi\cdot b(x,t,u,g(t))+\frac{1}{2}\sum_{\beta=1}^4\big(\sigma_{\beta}(x,t,u,g(t))\big)^2\,\partial^2_{u^\beta u^\beta}\varphi=0 \quad \text{on $(0,t_0)\times\R_+^4$},
        \\
        & \nabla_u\varphi(t,u)\cdot n_{\partial(\R^4_+)}(u)=0 \quad \text{on $(0,t_0)\times\partial(\R_+^4)$ },
        \\
        & \varphi(t_0,u)=\Phi(u)\quad \text{on $\{t_0\}\times\R_+^4$},
    \end{split}
\end{align*}
where we let $t_0\in\R^+$ and $\Phi\in C_c^2(\R_+^4)$ be arbitrary. 
Such an equation is always solvable since we have the right sign of the diffusion term (see \cite{risken_1996_fokker_planck} for details).
Going back to \eqref{Well-posedness of the non-linear fokker--planck equations proof eq 5}, we know that $g$ as well is a solution of this equation and thus we must conclude that $g=L\circ S^\epsilon[g]$.
Now let $T>0$ be small enough so that the composition map $L\circ S^\epsilon$ is a contraction in $L^{\infty}(Q;C_T^2)$.
This implies that $g\in L^{\infty}(Q;C_T^2)$ is a fixed point of $L\circ S^\epsilon$ and hence it must coincide with $f^\epsilon$ over $[0,T]$.
Applying the same argument over subsequent intervals $[T,2T]$, $[2T,3T]$ and so forth proves the uniqueness statement.

In particular, given any two $\epsilon,\Tilde{\epsilon}>0$, we take $g=f^{\tilde{\epsilon}}$ and we conclude that $f^{\epsilon}=f^{\tilde{\epsilon}}$.
That is to say $f(x,t)\coloneqq\text{Law}_{\R^4}(\Bar{u}^{\epsilon}(x,t))$ is independent of $\epsilon$ and is the unique solution of the nonlinear Fokker--Planck equation.

Finally, we assume that $f_0\in C^{\alpha}(Q;\pr_2(\R^4))$ and we show that the corresponding solution satisfies $f\in C^{\alpha}(Q;\pr_2(C[0,T];\R^4))$.
The theory of Wasserstein distances (see e.g. \cite{Villani}) ensures that we can find a stochastic basis supporting the white noise $W$ and random variables $u(x,0)\in C^\alpha(Q;L^2(\Omega))$ such that $\law_{\R^4}(u(x,0))=f_0(x,du)$ for every $x\in Q$.
We fix $\epsilon=1$ and we consider the iteration maps \eqref{contraction map 2nd moments} defined via this stochastic basis and with these initial data.
In particular we have $L\circ S^\epsilon[f]=f$, and thus for any $x,y\in Q$ we obtain
\begin{equation}
    \w_2\left(C\left([0,T];\R^4\right)\right)(f(x,\cdot),f(y,\cdot))\leq\E\left[\sup_{t\in[0,T]}|S^\epsilon[f](x,t)-S^\epsilon[f](y,t)|^2\right].
\end{equation}
This and formula \eqref{strong existence uniqueness McKean--Vlasov proof eq 12} with $\epsilon=1$ show that $f\in C^{\alpha}(Q;\pr_2(C[0,T];\R^4))$.
\end{proof}


\section{Error estimates between the particle system and the limiting model}
\label{Comparison between the particle system and the limiting model}

In this section we rigorously show that the limiting behaviour of the particle system \eqref{abstract 4 model particle system} as $M,N\to\infty$ is described by the McKean--Vlasov equation \eqref{abstract 4 model McKean--Vlasov} as stated in Theorem \ref{Mean squared error estimates for actual particles vs McKean--Vlasov particles} by obtaining an error estimate.
We will use the so-called Sznitman coupling method (cf. \cite{Sznitman1991}).

First, we lay out the right setting so as to get the convergence result. 
We fix a probability space $(\Omega,\mathcal{F},\p)$ and assume it supports all the random variables listed below.
First, for each $k\in\N$, let $\{W_k(x,t)\}_{k\in\N}$ be independent $4$-dimensional space-time white noise terms over $Q\times[0,\infty)$.

For any $\epsilon>0$ we then convolve and rescale the noise terms to obtain the $\epsilon$-correlated noise $W_k^{\epsilon}$ as in formula \eqref{formula/epsilon rescaled white noise}.
For $h\in\N$, we assume i.i.d. families of random initial conditions $u_h(x,0)\in C^{\alpha}(Q;L^2(\Omega))$ on the sheet $Q$. Moreover, we require them to be independent of the white noise terms $\{W_k(x,t)\}_{k\in\N}$.
Finally, as noted in Section \ref{introduction}, we take points $X_1,\dots,X_N\in Q$ in the center of the squares of an equispaced grid on $Q=[0,1]^d$ with side length $N^{-\frac{1}{d}}$.
We denote by $Q^N_i$ the square with center $X_i$, and we notice that $\meas(Q_i^N)=\frac{1}{N}$ and $\text{diam}(Q_i^N)=\sqrt{d}N^{-\frac{1}{d}}$.

We finally introduce the particles for the coupling method.
For $i=1,\dots,N$ and $k=1,\dots,M$, let $u_{ik}^{\epsilon}(t)$ be the solution of the particle system \eqref{abstract 4 model particle system} with initial data $u_{ik}(0)\coloneqq u_k(X_i,0)$ and Brownian motions $W_{ik}^{\epsilon}(t)\coloneqq W^{\epsilon}_k(X_i,t)$.
Let $\bar{u}_k^{\epsilon}(x,t)$ be the solution of the McKean--Vlasov equation with initial data $u_k(x,0)$ and correlated noise $W_k^{\epsilon}(x,t)$, and for $i=1,\dots,N$ define $\bar{u}_{ik}^{\epsilon}(x,t)\coloneqq \bar{u}^{\epsilon}_k(X_i,t)$.

Owing the i.i.d. properties of the initial data and the noise terms, we have the following.
\begin{lemma}
For fixed $i$, the particles $u^{\epsilon}_{ik}(t)$ are exchangeable for $k=1,\dots,M$.
Moreover, for fixed $i$, the particles $\bar{u}_{ik}^{\epsilon}(t)$ are i.i.d. for $k\in\N$.
\end{lemma}

We point out that this is not the case for the index $i$, both for the particles $u_{ik}^{\epsilon}$ and $\bar{u}^{\epsilon}_{ik}$.
Indeed, the laws of $u_{ik}^{\epsilon}$ and $u_{jk}^{\epsilon}$, or $\bar{u}_{ik}^{\epsilon}$ and $\bar{u}_{ik}^{\epsilon}$ respectively, might differ as a result of the $x$ dependence of their defining equations.
Furthermore, even if the points $X_i,X_j\in Q$ are far from each other, namely if $|X_i-X_j|>2\epsilon$, so that their noise terms $W^{\epsilon}_k(X_i,t)$ and $W^{\epsilon}_k(X_j,t)$ are independent, the particles might still be correlated as a result of their initial data.
In fact, from the point of view of modelling in neuroscience, we expect $u_k(x,0)$ to be close to $u_k(y,0)$ for $x$ close to $y$.

We are finally ready to prove the convergence result of Theorem \ref{Mean squared error estimates for actual particles vs McKean--Vlasov particles}.
We first stress the following.

\begin{remark}
As mentioned in Section \ref{main results}, we point out that we do not need to impose any constraint on the ratio between the correlation radius $\epsilon$ of the noise and the minimum distance $\sqrt{d}N^{-\nicefrac{1}{d}}$ between two grid points $X_i,X_j\in Q$.
The choice of the scaling regime $(\epsilon, N)$, with $N\to\infty$ and $\epsilon\to0$ or possibly also $\epsilon\equiv\epsilon_0$ a constant, is purely arbitrary and dictated by modelling arguments only.
One might impose $\epsilon\,N^{\frac{1}{d}}<\sqrt{d}$ so that all the particles sense independent noise, or choose to impose a certain ratio $\epsilon\, N^{\frac{1}{d}}>\sqrt{d}$, so that neurons at locations close enough to each other sense correlated noise.
The results and the proof of Theorem \ref{Mean squared error estimates for actual particles vs McKean--Vlasov particles} are unchanged.
\end{remark}
\begin{proof}[\textbf{Proof of Theorem \ref{Mean squared error estimates for actual particles vs McKean--Vlasov particles}}]
\noindent
For any $i=1,\dots,N$ and $k=1,\dots,M$, take the difference $|u^\epsilon_{ik}-\Bar{u}^\epsilon_{ik}|$ between actual particles and McKean--Vlasov particles.
Applying the It\^o formula and exploiting the respective equations \eqref{abstract 4 model particle system} and \eqref{abstract 4 model McKean--Vlasov}, we get 
\begin{align}\label{Mean squared error estimates for actual particles vs McKean--Vlasov particles proof eq 1} 
    \begin{split}
    |u_{ik}^\epsilon(t)-\Bar{u}^\epsilon_{ik}(t)|^2=\,&
    2\int_0^t\big(u^\epsilon_{ik}(r)-\Bar{u}^\epsilon_{ik}(r)\big)\big(b(X_i,r,u^\epsilon_{ik}(r),f^\epsilon_{MN}(r))-b(X_i,r,\Bar{u}^\epsilon_{ik}(r),f(r))\big)\,dr
    \\
    &+2\int_0^t\!\!\!\big(u_{ik}^\epsilon(r)-\Bar{u}_{ik}^\epsilon(r)\big)\big(\sigma(X_i,r,u_{ik}^\epsilon,f^\epsilon_{MN})-\sigma(X_i,r,\Bar{u}_{ik}^\epsilon,f)\big)\,dW^\epsilon(X_i,r)
    \\
    &+2\int_0^t\big(u_{ik}^\epsilon(r)-\Bar{u}_{ik}^\epsilon(r)\big)\big(d\Bar{\ell}_{ik}(r)-d\ell_{ik}(r)\big)
    \\
    &+\int_0^t\!\!\!\big(\sigma(X_i,r,u_{ik}^\epsilon(r),f^\epsilon_{MN}(r))-\sigma(X_i,r,\Bar{u}^\epsilon_{ik}(r),f(r))\big)^2\,dr.
    \end{split}
\end{align}
Now we argue as in \eqref{strong existence and uniqueness 2nd moments proof 2}--\eqref{strong existence and uniqueness 2nd moments proof 5}.
First we drop the third term in \eqref{Mean squared error estimates for actual particles vs McKean--Vlasov particles proof eq 1}, which is always negative owing to the definition of the reflection terms $\ell_{ik}$ and $\Bar{\ell}_{ik}$.
Then we take the supremum in $t\in[0,\tau]$ and apply the expectation. 
Next we use the Burkholder--Davis--Gundy and H\"older's inequality, we absorb the necessary terms into the left hand side and finaly we exploit Gr\"onwall's Lemma.
We eventually obtain, for $C=C(T)$,
\begin{align}\label{Mean squared error estimates for actual particles vs McKean--Vlasov particles proof eq 2}
    \begin{split}
    \E\left[\sup_{t\in[0,\tau]}|u_{ik}^\epsilon(t)-\Bar{u}^\epsilon_{ik}(t)|^2\right]
    \leq
    C\Bigg(&\int_0^{\tau}\E\left[|b(X_i,r,u^\epsilon_{ik},f^\epsilon_{MN})-b(X_i,r,\Bar{u}^\epsilon_{ik},f)|^2\,\right]\,dr
    \\
    &+\int_0^{\tau}\E\left[|\sigma(X_i,r,u^\epsilon_{ik},f^\epsilon_{MN})-\sigma(X_i,r,\Bar{u}^\epsilon_{ik},f)|^2\,\right]\,dr\Bigg).
\end{split}
\end{align}

In order to split the terms on the right hand side of the inequality \eqref{Mean squared error estimates for actual particles vs McKean--Vlasov particles proof eq 2} and exploit the particular structure of the drift and diffusion terms, we introduce the following probability measure on $Q\times\R^4$:
\begin{equation}\label{Mean squared error estimates for actual particles vs McKean--Vlasov particles proof eq 3}
    \Bar{f}^\epsilon_{MN}(t,dy,dv)=\frac{1}{MN}\sum_{j=1}^N\sum_{m=1}^M\delta_{(X_j,\Bar{u}^\epsilon_{jm}(t))}\in\pr(Q\times\R^4).
\end{equation}
This measure is just the empirical measure associated to the collection of McKean--Vlasov particles $(X_j,\Bar{u}^\epsilon_{jm}(t))$.
We have
\begin{align}\label{Mean squared error estimates for actual particles vs McKean--Vlasov particles proof eq 4}
    |b(X_i,r,u^\epsilon_{ik},f^\epsilon_{MN})-b(X_i,r,\Bar{u}^\epsilon_{ik},f)|
    \leq &\,
    |b(X_i,r,u^\epsilon_{ik},f^\epsilon_{MN})-b(X_i,r,\Bar{u}^\epsilon_{ik},f^\epsilon_{MN})| \nonumber
    \\
    &
    +
    |b(X_i,,r,\Bar{u}^\epsilon_{ik},f^\epsilon_{MN})-b(X_i,r,\Bar{u}^\epsilon_{ik},\Bar{f}^\epsilon_{MN})|
    \\
    &
    +
    |b(X_i,,r,\Bar{u}^\epsilon_{ik},\Bar{f}^\epsilon_{MN})-b(X_i,r,\Bar{u}^\epsilon_{ik},f)|. \nonumber
\end{align}
Due to the structure of the drift term \eqref{drift term structure} and its Lipschitz properties \eqref{b_0 sigma_0 lipschitz property}--\eqref{b_1 sigma_1 lipschitz property}, and owing to the definition of $\Bar{f}^\epsilon_{MN}(r)$ in \eqref{Mean squared error estimates for actual particles vs McKean--Vlasov particles proof eq 3}, we get the following estimates for terms on the right hand side of \eqref{Mean squared error estimates for actual particles vs McKean--Vlasov particles proof eq 4}:
\begin{align}\label{Mean squared error estimates for actual particles vs McKean--Vlasov particles proof eq 5}
    \big|b(X_i,r,u^\epsilon_{ik},f^\epsilon_{MN})- b(X_i,r,\Bar{u}^\epsilon_{ik},f^\epsilon_{MN})\big|
    & \leq
    C\,\big|u^\epsilon_{ik}(r)-\Bar{u}^\epsilon_{ik}(r)\big|,
    \nonumber \\[2mm]
    \big|b(X_i,r,\Bar{u}^\epsilon_{ik},f^\epsilon_{MN})-  b(X_i,r,\Bar{u}^\epsilon_{ik},\Bar{f}^\epsilon_{MN})\big| 
    &
    \leq C\,\Bigg|\int_{Q\times\R^4}b_1(X_i,y,r,\Bar{u}^\epsilon_{ik},v)f^\epsilon_{MN}(r,dy,dv) 
     \nonumber \\
    & \qquad\quad -\int_{Q\times\R^4}b_1(X_i,y,r,\Bar{u}^\epsilon_{ik},v)\Bar{f}^\epsilon_{MN}(r,dy,dv)\Bigg|  
    \nonumber \\
    &
    \leq
    C\,\frac{1}{MN}\sum_{j=1}^N\sum_{m=1}^M\big|u^\epsilon_{jm}(r)-\Bar{u}^\epsilon_{jm}(r)\big|
    ,
\\[2mm]
    \big|b(X_i,r,\Bar{u}^\epsilon_{ik},\Bar{f}^\epsilon_{MN})-  b(X_i,r,\Bar{u}^\epsilon_{ik},f)\big|
    &
    \leq
    C\,\Bigg|\int_{Q\times\R^4}b_1(X_i,y,r,\Bar{u}^\epsilon_{ik},v)\Bar{f}^\epsilon_{MN}(r,dy,dv)
    \nonumber \\ 
    & \qquad\quad-\int_{Q\times\R^4}b_1(X_i,y,r,\Bar{u}^\epsilon_{ik},v)f(r,dy,dv)\Bigg| 
    \\
    &
    =
    C\,\Bigg|\frac{1}{MN}\!\!\sum_{j=1}^N\sum_{m=1}^M\!\bigg(b_1(X_i,X_j,r,\Bar{u}^\epsilon_{ik}(r),\Bar{u}^\epsilon_{jm}(r))\!
    \\
     &\qquad\qquad\quad-\!\int_{Q\times\R^4}b_1(X_i,y,r,\Bar{u}^\epsilon_{ik},v)f(r,dy,dv)\bigg)\,\Bigg|
     ,
\end{align}
for a constant $C=C(b)$ only depending on the Lipschitz constants of $b$.
An identical splitting \eqref{Mean squared error estimates for actual particles vs McKean--Vlasov particles proof eq 4} holds for the term $(\sigma(X_i,r,u^\epsilon_{ik},f^\epsilon_{MN})-\sigma(X_i,r,\Bar{u}^\epsilon_{ik},f))$ and using the Lipschitz properties \eqref{b_0 sigma_0 lipschitz property}--\eqref{b_1 sigma_1 lipschitz property} of $\sigma$ we obtain analogous estimates to \eqref{Mean squared error estimates for actual particles vs McKean--Vlasov particles proof eq 5}.

Going back to \eqref{Mean squared error estimates for actual particles vs McKean--Vlasov particles proof eq 2}, we exploit \eqref{Mean squared error estimates for actual particles vs McKean--Vlasov particles proof eq 4} and \eqref{Mean squared error estimates for actual particles vs McKean--Vlasov particles proof eq 5}.
After standard convexity inequalities we obtain, for $C=C(T,b,\sigma)$,
\begin{align}\label{Mean squared error estimates for actual particles vs McKean--Vlasov particles proof eq 6}
    \E\!\left[\!\sup_{t\in[0,\tau]}\!|u^\epsilon_{ik}(t)-\Bar{u}^\epsilon_{ik}(t)|^2\!\right]
    \!\leq\,&
    C
    \Bigg\{\!\!\int_0^{\tau}\!\!\!\!\E\big[|u^\epsilon_{ik}(r)-\Bar{u}^\epsilon_{ik}(r)|^2\big]\,dr
    \nonumber \\
    &\quad+
    \int_0^{\tau}\!\!\frac{1}{MN}\!\!\sum_{j,m=1}^{N,M}\E\big[|u^\epsilon_{jm}(r)-\Bar{u}^\epsilon_{jm}(r)|^2\big]\,dr 
   \nonumber \\
    &\quad+
    \int_0^{\tau}\!\!\!\E\Bigg[\bigg|\frac{1}{MN}\sum_{j,m=1}^{N,M}\bigg(b_1(X_i,X_j,r,\Bar{u}^\epsilon_{ik},\Bar{u}^\epsilon_{jm})
     \\
    &\qquad\qquad-\!\!\int_{Q\times\R^4}\!\!\!\!\!\!\!\!\!\!\!\!b_1(X_i,y,r,\Bar{u}^\epsilon_{ik},v)f(r,dy,dv)\bigg)\bigg|^2\Bigg]dr
    \Bigg\}
    \\
    &\quad+
    \int_0^{\tau}\!\!\!\E\Bigg[\bigg|\frac{1}{MN}\sum_{j,m=1}^{N,M}\bigg(\sigma_1(X_i,X_j,r,\Bar{u}^\epsilon_{ik},\Bar{u}^\epsilon_{jm})
    \\
    &\qquad\qquad-\!\!\int_{Q\times\R^4}\!\!\!\!\!\!\!\!\!\!\!\!\sigma_1(X_i,y,r,\Bar{u}^\epsilon_{ik},v)f(r,dy,dv)\bigg)\bigg|^2\Bigg]dr
    \Bigg\}.
\end{align}
Averaging \eqref{Mean squared error estimates for actual particles vs McKean--Vlasov particles proof eq 6} over $i=1,\dots,N$ and $k=1,\dots,M$, and then using Gr\"onwall's Lemma to get rid of the first two terms on the right hand side, we obtain 
\begin{gather}
\begin{aligned}\label{Mean squared error estimates for actual particles vs McKean--Vlasov particles proof eq 7}
    \frac{1}{MN}\sum_{i=1}^N\sum_{k=1}^M\E\!\left[\!\sup_{t\in[0,\tau]}\!|u^\epsilon_{ik}(t)-\Bar{u}^\epsilon_{ik}(t)|^2\!\right]
    \!\leq\,&
    C
    \frac{1}{MN}\sum_{i=1}^N\sum_{k=1}^M\int_0^{\tau}\!\! R^b_{ik}(t)+R^{\sigma}_{ik}(t)\,\,dt,
\end{aligned}
\end{gather}
for another constant $C=C(T,b,\sigma)$. 
Here we have defined 
\begin{align}
    R^b_{ik}\!(t)\!=&\,\E\Bigg[\bigg|\frac{1}{MN}\sum_{j=1}^N\sum_{m=1}^M\Big(b_1(X_i,X_j,t,\Bar{u}^\epsilon_{ik}(t),\Bar{u}^\epsilon_{jm}(t))\!-\!\int_{Q\times\R^4}\!\!\!\!\!\!\!\!\!\!\!\!b_1(X_i,y,t,\Bar{u}^\epsilon_{ik}(t),v)f(t,y,dv)\,dy\Big)\,\bigg|^2\Bigg],
    \label{Mean squared error estimates for actual particles vs McKean--Vlasov particles proof eq 8.1}
    \\
    R^{\sigma}_{ik}\!(t)\!=&\,\E\Bigg[\bigg|\frac{1}{MN}\sum_{j=1}^N\sum_{m=1}^M\!\!\!\Big(\sigma_1(X_i,X_j,t,\Bar{u}^\epsilon_{ik}(t),\Bar{u}^\epsilon_{jm}(t))\!-\!\int_{Q\times\R^4}\!\!\!\!\!\!\!\!\!\!\!\!\sigma_1(X_i,y,t,\Bar{u}^\epsilon_{ik}(t),v)f(t,y,dv)\,dy\Big)\,\bigg|^2\!\Bigg],
    \label{Mean squared error estimates for actual particles vs McKean--Vlasov particles proof eq 8.2}
\end{align}
which are the arguments of the last two integrals on the right hand side of \eqref{Mean squared error estimates for actual particles vs McKean--Vlasov particles proof eq 6}.
Heuristically, the error terms $R_{ik}^b$ and $R_{ik}^{\sigma}$ should be small in view of the weak law of large numbers. Indeed, upon conditioning on $\Bar{u}^\epsilon_{ik}$, for each fixed $j=1,\dots,N$, we are essentially taking the average of the i.i.d. terms $b_1(X_i,X_j,t,\Bar{u}^\epsilon_{ik}(t),\Bar{u}^\epsilon_{jm}(t))$ for $m=1,\dots,M$, and then subtracting their common expectation $\int_{Q\times\R^4}b_1(X_i,y,r,\Bar{u}^\epsilon_{ik}(t),v)f(t,y,dv)\,dy$.

In order to control the term to the right in \eqref{Mean squared error estimates for actual particles vs McKean--Vlasov particles proof eq 7}, we need the following estimate whose proof is postponed for the sake of the reader.
For any $T>0$, we have
\begin{equation}\label{R^b estimate}
    \sup_{t\in[0,T]}|R^b_{ik}(t)|+|R^\sigma_{ik}(t)|\leq C\left(1+\sup_{x\in Q}\E\left[|u_k(x,0)|^2\right]\right)\,\left(\frac{1}{M}+\frac{1}{N^{\frac{\alpha}{d}}}\right),
\end{equation}
for a constant $C=C(T,b,\sigma,\rho,[u(\cdot,0)]_{\alpha})$, for every $i=1,\dots,N$ and $k=1,\dots,M$.
Plugging \eqref{R^b estimate} into \eqref{Mean squared error estimates for actual particles vs McKean--Vlasov particles proof eq 7} we obtain, for $C=C(T,b,\sigma,\rho,[u(\cdot,0)]_{\alpha})$,
\begin{gather}
\begin{aligned}\label{Mean squared error estimates for actual particles vs McKean--Vlasov particles proof eq 9}
    \frac{1}{MN}\sum_{i=1}^N\sum_{k=1}^M\E\!\left[\!\sup_{t\in[0,\tau]}\!|u^\epsilon_{ik}(t)-\Bar{u}^\epsilon_{ik}(t)|^2\!\right]
    \!\leq\,&
    C\left(1+\sup_{x\in Q}\E\left[|u_k(x,0)|^2\right]\right)\,\left(\frac{1}{M}+\frac{1}{N^{\frac{\alpha}{d}}}\right).
\end{aligned}
\end{gather}

We can now finally prove Theorem \ref{Mean squared error estimates for actual particles vs McKean--Vlasov particles}.
We go back to \eqref{Mean squared error estimates for actual particles vs McKean--Vlasov particles proof eq 6}, and get rid of the first term on the right hand side with Gr\"onwall's Lemma.
We control the second term on the right hand side with \eqref{Mean squared error estimates for actual particles vs McKean--Vlasov particles proof eq 9} and the last two terms with \eqref{R^b estimate}.
This yields formula \eqref{Mean squared error estimates for actual particles vs McKean--Vlasov particles eq 1} and concludes the proof.

\end{proof}

\begin{proof}[\textbf{Proof of estimate \eqref{R^b estimate}}]
We prove the estimate for $R_{ik}^b$.
Identical computations replacing $b$ with $\sigma$ prove the analogous result for $R^{\sigma}_{ik}$.
Recalling that $\meas(Q_j^N)=\frac{1}{N}$, we split the term as
\begin{align}\label{R^b estimate proof eq 1}
    R^b_{ik}\!(t)\!=&\,\E\Bigg[\bigg|\frac{1}{MN}\sum_{j=1}^N\sum_{m=1}^M\Big(b_1(X_i,X_j,t,\Bar{u}^\epsilon_{ik}(t),\Bar{u}^\epsilon_{jm}(t))\!-\!\int_{Q\times\R^4}\!\!\!\!\!\!\!\!\!\!\!\!b_1(X_i,y,t,\Bar{u}^\epsilon_{ik}(t),v)f(t,y,dv)\,dy\Big)\,\bigg|^2\Bigg]
    \nonumber \\
    \leq \,
    &
    \E\Bigg[\bigg|\frac{1}{MN}\sum_{j=1}^N\sum_{m=1}^M\Big(b_1(X_i,X_j,t,\Bar{u}^\epsilon_{ik}(t),\Bar{u}^\epsilon_{jm}(t))\!-\!\int_{\R^4}\!\!\!\!b_1(X_i,X_j,t,\Bar{u}^\epsilon_{ik}(t),v)f(t,X_j,dv)\Big)\,\bigg|^2\Bigg]
    \\
    &
    +
    \E\Bigg[\bigg|\sum_{j=1}^N\Big(\int_{Q_j^N}\!\int_{\R^4}\!\!b_1(X_i,X_j,t,\Bar{u}^\epsilon_{ik}(t),v)\,f(t,X_j,dv)\! \nonumber \\
     &\qquad -\!\int_{\R^4}\!\!\!\!\!b_1(X_i,y,t,\Bar{u}^\epsilon_{ik}(t),v)f(t,y,dv)\,\,dy\Big)\,\bigg|^2\Bigg]. \nonumber
\end{align}

For the first term of \eqref{R^b estimate proof eq 1}, the estimate is proved similarly to the weak law of large numbers.
Indeed, for $C=C(T,b,\sigma)$, we compute
\begin{align}\label{R^b estimate proof eq 2}
    \E&\Bigg[\bigg|\frac{1}{MN}\sum_{j=1}^N\sum_{m=1}^M\Big(b_1(X_i,X_j,t,\Bar{u}^\epsilon_{ik}(t),\Bar{u}^\epsilon_{jm}(t))\!-\!\int_{\R^4}\!\!\!\!b_1(X_i,X_j,t,\Bar{u}^\epsilon_{ik}(t),v)f(t,X_j,dv)\Big)\,\bigg|^2\Bigg]
    \nonumber \\
    &\leq
    \frac{1}{N}\sum_{j=1}^N
    \E\Bigg[\bigg|\frac{1}{M}\sum_{m=1}^M\Big(b_1(X_i,X_j,t,\Bar{u}^\epsilon_{ik}(t),\Bar{u}^\epsilon_{jm}(t))\!-\!\int_{\R^4}\!\!\!\!b_1(X_i,X_j,t,\Bar{u}^\epsilon_{ik}(t),v)f(t,X_j,dv)\Big)\,\bigg|^2\Bigg]
    \nonumber \\
    &=
    \frac{1}{N}\sum_{j=1}^N\frac{1}{M^2}\sum_{\substack{m_1=1\\m_2=1}}^M
    \E\Bigg[\Big(b_1(X_i,X_{j},t,\Bar{u}^\epsilon_{ik},\Bar{u}^\epsilon_{jm_1})-\!\int_{\R^4}\!\!\!\!b_1(X_i,X_j,t,\Bar{u}^\epsilon_{ik},v)f(t,X_j,dv)\Big)
    \\
    &\qquad\qquad\qquad\qquad\qquad\cdot\Big(b_1(X_i,X_{j},t,\Bar{u}^\epsilon_{ik},\Bar{u}^\epsilon_{jm_2})\!-\!\int_{\R^4}\!\!\!\!b_1(X_i,X_j,t,\Bar{u}^\epsilon_{ik},v)f(t,X_j,dv)\Big)\Bigg]
    \\
    &=
    \frac{1}{N}\sum_{j=1}^N\frac{1}{M^2}\sum_{m=1}^M
    \E\Bigg[\bigg(b_1(X_i,X_{j},t,\Bar{u}^\epsilon_{ik},\Bar{u}^\epsilon_{jm})-\!\int_{\R^4}\!\!\!\!b_1(X_i,X_j,t,\Bar{u}^\epsilon_{ik},v)f(t,X_j,dv)\bigg)^2\Bigg]
    \\
    &\leq
    \frac{1}{M}\, C\,\left(1+\sup_{x\in Q}\E\left[|u_k(x,0)|^2\right]\right).
\end{align}
In the first passage we used a convexity inequality.
In the last passage we used the sublinear growth properties \eqref{b_1 sigma_1 sublinear property} of $b_1$ and the a priori estimate \eqref{Strong existence and uniqueness for the McKean--Vlasov eq 1} for McKean--Vlasov particles.
In the second passage we unfolded the square, and in the third we noticed that, after conditioning with respect to $\Bar{u}^\epsilon_{ik}(t)$, only the ``diagonal terms'' survive in the sum, i.e. those with $m_1=m_2$.
Namely, when $m_1\neq m_2$ the corresponding term in \eqref{R^b estimate proof eq 2} is identically zero.
Indeed, under this condition, assuming by symmetry $m_1\neq k$, we have that $\Bar{u}^\epsilon_{jm_1}(t)$ is independent of $\Bar{u}^\epsilon_{jm_2}(t)$ and $\Bar{u}^\epsilon_{ik}(t)$.
Hence we compute
\begin{align}\label{R^b estimate proof eq 3}
    \E&\Bigg[\Big(b_1(X_i,X_{j},t,\Bar{u}^\epsilon_{ik},\Bar{u}^\epsilon_{jm_1})-\!\int_{\R^4}\!\!\!\!b_1(X_i,X_j,t,\Bar{u}^\epsilon_{ik},v)f(t,X_j,dv)\Big)
    \nonumber \\
    &\qquad\qquad\qquad\qquad\qquad\cdot\Big(b_1(X_i,X_{j},t,\Bar{u}^\epsilon_{ik},\Bar{u}^\epsilon_{jm_2})\!-\!\int_{\R^4}\!\!\!\!b_1(X_i,X_j,t,\Bar{u}^\epsilon_{ik},v)f(t,X_j,dv)\Big)\Bigg]
    \nonumber \\
    &=
    \E\Bigg[\E\bigg[\Big(b_1(X_i,X_{j},t,\Bar{u}^\epsilon_{ik},\Bar{u}^\epsilon_{jm_1})-\!\int_{\R^4}\!\!\!\!b_1(X_i,X_j,t,\Bar{u}^\epsilon_{ik},v)f(t,X_j,dv)\Big)
    \\
    &\qquad\qquad\qquad\qquad\cdot\Big(b_1(X_i,X_{j},t,\Bar{u}^\epsilon_{ik},\Bar{u}^\epsilon_{jm_2})\!-\!\int_{\R^4}\!\!\!\!b_1(X_i,X_j,t,\Bar{u}^\epsilon_{ik},v)f(t,X_j,dv)\Big)\bigg|\,\,\Bar{u}^\epsilon_{ik}\,\bigg]\Bigg]
    \\
    &=
    \E\Bigg[\E\bigg[\Big(b_1(X_i,X_{j},t,u,\Bar{u}^\epsilon_{jm_1})-\!\int_{\R^4}\!\!\!\!b_1(X_i,X_j,t,u,v)f(t,X_j,dv)\Big)\bigg]_{u=\bar{u}_{ik}^\epsilon}
    \\
    &\qquad\qquad\qquad\qquad\cdot\E\bigg[\Big(b_1(X_i,X_{j},t,\Bar{u}^\epsilon_{ik},\Bar{u}^\epsilon_{jm_2})\!-\!\int_{\R^4}\!\!\!\!b_1(X_i,X_j,t,\Bar{u}^\epsilon_{ik},v)f(t,X_j,dv)\Big)\bigg|\,\,\Bar{u}^\epsilon_{ik}\,\bigg]\Bigg]
    \\
    &=0.
\end{align}
In the second passage we conditioned on $\bar{u}_{ik}^\epsilon$ and in the third passage we used standard properties of the conditional expectation (see e.g. \cite[Chapter 2]{daprato_zabczyk_1992}).
Finally we used that $E[b_1(X_i,X_j,t,u,\Bar{u}^\epsilon_{jm})]=\int_{\R^4}b_1(X_i,X_j,t,u,v)f(t,X_j,dv)$ by definition of $f(t,X_j,dv)$.

For the second term second term on the right hand side of \eqref{R^b estimate proof eq 1}, we first compute

\begin{align}\label{R^b estimate proof eq 4}
    \bigg|\int_{\R^4}&\!\!b_1(X_i,X_j,t,\Bar{u}^\epsilon_{ik}(t),v)\,f(t,X_j,dv)\!-\!\int_{\R^4}\!\!\!\!\!b_1(X_i,y,t,\Bar{u}^\epsilon_{ik}(t),v)f(t,y,dv)\bigg|
    \nonumber \\
    &
    \leq
    \int_{(\R^4)^2}\!\!\left|b_1(X_i,X_j,t,\Bar{u}^\epsilon_{ik}(t),v)\!-\!b_1(X_i,y,t,\Bar{u}^\epsilon_{ik}(t),w)\right|\,\pi_0(X_j,y,dv,dw)
    \nonumber \\
    &
    \leq  C\,
     \int_{(\R^4)^2}\!\!\!\!\left|X_j-y\right|^\alpha+|v-w|\,\,\pi_0(X_j,y,dv,dw)
     \\
     &
     \leq C\,
     |X_j-y|^\alpha+\w_1\left(\R^4\right)(f(t,X_j,dv),f(t,y,dv))
     \\
     &\leq C\,|X_j-y|^\alpha
\end{align}
for a constant $C=C(T,b,\sigma,\rho,[u(\cdot,0)]_{\alpha})$.
In the second passage we took any optimal pairing $\pi_0(X_j,y,dv,dw)$ for $\w_1(f(t,X_j,dv),f(t,y,dv))$, in the third we used the Lipschitz and H\"older properties \eqref{b_1 sigma_1 lipschitz property} of $b_1$, and in the last we used the ordering $\w_1\leq\w_2$ of Wasserstein distances and the H\"older continuity \eqref{Well-posedness of the non-linear fokker--planck equations eq 2} of $f$ in $\w_2$.
Then, using \eqref{R^b estimate proof eq 4} and recalling that $\meas(Q_j^N)=\frac{1}{N}$ and $\text{diam}(Q_j^N)=N^{-\frac{1}{d}}$, we compute

\begin{align}\label{R^b estimate proof eq 5}
    \E&\Bigg[\bigg|\sum_{j=1}^N\Big(\int_{Q_j^N}\!\int_{\R^4}\!\!b_1(X_i,X_j,t,\Bar{u}^\epsilon_{ik}(t),v)\,f(t,X_j,dv)\!-\!\int_{\R^4}\!\!\!\!\!b_1(X_i,y,t,\Bar{u}^\epsilon_{ik}(t),v)f(t,y,dv)\,\,dy\Big)\,\bigg|^2\Bigg]
    \nonumber \\
    &
    \leq
    \E\Bigg[\bigg|\sum_{j=1}^N\int_{Q_j^N}\!\!C|X_j-y|^\alpha\,dy\bigg|^2\Bigg]
    \\
    &
    \leq 
    C\,\text{diam}(Q_j^N)^{2\alpha}
    \\
    &=C\,N^{-\frac{2\alpha}{d}},
\end{align}
for a constant $C=C(T,b,\sigma,\rho,[u(\cdot,0)]_{\alpha})$.

In conclusion, combining \eqref{R^b estimate proof eq 1} with estimates \eqref{R^b estimate proof eq 3} and \eqref{R^b estimate proof eq 5} we obtain the estimate \eqref{R^b estimate}.
\end{proof}


\section{Convergence of empirical measures}
\label{Convergence of empirical measures}

In this last section, we further analyze the limiting behaviour of the particle system as we let $M,N\to\infty$ and prove Theorem \ref{Rate of convergence for empirical measures}.
In the same setting outlined in Section \ref{Comparison between the particle system and the limiting model}, we show that the time dependent empirical measure 
$$
f^{\epsilon}_{MN}(t,dx,du)=\frac{1}{MN}\sum_{j=1}^N\sum_{m=1}^M\delta_{(X_j,u^{\epsilon}_{jm}(t))}\in\pr(Q\times\R^4),
$$ 
associated to the particle system \eqref{abstract 4 model particle system}, located at the grid points $X_1,\dots,X_N$, converges in Wasserstein distance $\w_1(Q\times\R^4)$ to the measure $f(t,dx,du)$, obtained from the solution of the Fokker--Planck equation \eqref{abstract 4 model nonlinear fokker--planck} via formula \eqref{McKean--Vlasov law on Q times R^4}.
The key step towards the result is to split the Wasserstein distance:
\begin{align}\label{splitting wasserstein distance}
    \w_1(Q\times\R^4)(f^{\epsilon}_{MN}(t),f(t))
    \leq
    \w_1(f^{\epsilon}_{MN}(t),\Bar{f}^{\epsilon}_{MN}(t))+
    \w_1(\Bar{f}^{\epsilon}_{MN}(t),\Bar{f}_N(t))+
    \w_1(\Bar{f}_N(t),f(t)).
\end{align}
Here $\bar{f}^{\epsilon}_{MN}=\frac{1}{MN}\sum_{j=1}^N\sum_{m=1}^M\delta_{(X_j,\Bar{u}^{\epsilon}_{jm}(t))}$ is the empirical measure of the associated McKean--Vlasov particles as in Section \ref{Comparison between the particle system and the limiting model}, and
\begin{equation}
    \bar{f}_N(t,dx,du)\coloneqq\frac{1}{N}\sum_{j=1}^N\delta_{X_j}\otimes f(t,X_j,du),
\end{equation}
an auxiliary measure, can be viewed as a Riemann sum approximation for the measure $f(t,dx,du)$.
Then Theorem \ref{Rate of convergence for empirical measures} is an immediate consequence of the splitting \eqref{splitting wasserstein distance} and Lemma \ref{rate of convergence f_MN vs bar f_MN}, \ref{rate of convergence bar f_MN vs bar f_N} and \ref{rate of convergence bar f_N vs f} below. The first term in \eqref{splitting wasserstein distance} is readily handled with Theorem \ref{Mean squared error estimates for actual particles vs McKean--Vlasov particles} as follows.

\begin{lemma}\label{rate of convergence f_MN vs bar f_MN}
In the setting above, for any $T>0$ we have
\begin{equation}\label{rate of convergence f_MN vs bar f_MN eq 1}
    \sup_{t\in[0,T]}\E\left[\w_1(Q\times\R^4)(f^{\epsilon}_{MN}(t),\Bar{f}^{\epsilon}_{MN}(t))\right]\leq C  \left(1+\sup_{x\in Q}\E\left[|u_k(x,0)|^2\right]^{\frac{1}{2}}\right) \left(\frac{1}{M^{\frac{1}{2}}}+\frac{1}{N^{\frac{\alpha}{d}}}\right),
\end{equation}
for a constant $C=C(T,\rho,b,\sigma,[u(\cdot,0)]_{\alpha})$.
\end{lemma}

\begin{proof}
It suffices to notice that
    \begin{equation}\label{rate of convergence f_MN vs bar f_MN proof eq 1}
        \pi_0=\frac{1}{MN}\sum_{j=1}^N\sum_{k=1}^M\delta_{(X_j,u^{\epsilon}_{jk}(t),X_j,\Bar{u}^{\epsilon}_{jk}(t))},\nonumber
    \end{equation}
is an admissible pairing for $f^{\epsilon}_{MN}(t)$ and $\Bar{f}^{\epsilon}_{MN}(t)$. 
Then, by definition of the Wasserstein distance,
\begin{equation}\label{rate of convergence f_MN vs bar f_MN proof eq 2}\nonumber
    \w_1(f^{\epsilon}_{MN}(t),\Bar{f}^{\epsilon}_{MN}(t))
    \leq
    \int_{(Q\times\R^4)^2}\!\!\!\!\!\!\!\!\!\!\!\!\!\!\!\!|x-y|+|u-v|\,\,d\pi_0
    =
    \frac{1}{MN}\sum_{j=1}^N\sum_{m=1}^M|u^{\epsilon}_{jm}(t)-\Bar{u}^{\epsilon}_{jm}(t)|.
\end{equation}
Next we apply $\E$ at both sides of the inequality and use H\"older's inequality to get
\begin{gather}
\begin{aligned}
        \E\big[\w_1(f^{\epsilon}_{MN}(t),\Bar{f}^{\epsilon}_{MN}(t))\big]
        &
        \leq
        \frac{1}{MN}\sum_{j=1}^N\sum_{m=1}^M
        \E\big[|u^{\epsilon}_{jm}(t)-\Bar{u}^{\epsilon}_{jm}(t)|\big]
        \\
        \label{rate of convergence f_MN vs bar f_MN proof eq 3}
        &\leq 
        \frac{1}{MN}\sum_{j=1}^N\sum_{m=1}^M
        \E\left[|u_{jm}^{\epsilon}(t)-\Bar{u}^{\epsilon}_{jm}(t)|^2\right]^{\frac{1}{2}}.
\end{aligned}
\end{gather}
Now we conclude by plugging \eqref{Mean squared error estimates for actual particles vs McKean--Vlasov particles eq 1} into \eqref{rate of convergence f_MN vs bar f_MN proof eq 3}.
\end{proof}

Let us now turn to the second term in the splitting \eqref{splitting wasserstein distance}. 
To start with, in a weak law of large numbers manner, we get the following lemma.

\begin{lemma}\label{convergence bar f_MN vs bar f_N}
In the setting above, for every $N\in\N$ and every $\epsilon>0$, for any $t\geq0$,
\begin{equation}\label{convergence bar f_MN vs f eq 1}
    \lim_{M\to\infty}\w_1(Q\times\R^4)(\Bar{f}^{\epsilon}_{MN}(t),\bar{f}_N(t))=0\quad\text{in probability.}\nonumber
\end{equation}
\end{lemma}

\begin{proof}
The key observation is the following: if $\varphi(x,v)\in C(Q\times\R^4)$ has linear growth in $v$, that is $|\varphi(x,v)|\leq L_{\varphi} (1+|v|)$ for some constant $L_{\varphi}\geq0$, then we have $\langle\varphi,\Bar{f}_{MN}(t)\rangle\to\langle\varphi,\bar{f}_N(t)\rangle$ in $L^2(\Omega)$ as $M\to\infty$, uniformly in $N\in\N$.
Indeed, with the same arguments as in the proof of Lemma \ref{R^b estimate}, we have, for a constant $C=C(T,b,\sigma)$ independent of $\varphi$,
{\small
 \begin{align}\label{convergence bar f_MN vs f proof eq 1}
        \E\left[|\langle\varphi,\Bar{f}^{\epsilon}_{MN}(t)\rangle-\langle\varphi,\bar{f}_N(t)\rangle|^2\right]
        &\!\leq\,\frac{1}{N}\sum_{j=1}^N
        \E\Bigg[\bigg|\frac{1}{M}\sum_{m=1}^M\Big(\varphi(X_j,\Bar{u}^{\epsilon}_{jm}(t))\!-\!\!\int_{\R^4}\varphi(X_j,v)f(t,X_j,dv)\Big)\bigg|^2\Bigg]
        \\
        &\!=
        \frac{1}{N}\sum_{j=1}^N
        \frac{1}{M^2}\!\!\sum_{\substack{m_1\neq m_2}}\!\!\E\Bigg[\!\bigg(\!\varphi(X_{j},\Bar{u}^{\epsilon}_{jm_1}(t))\!-\!\!\int_{\R^4}\!\!\!\!\!\varphi(X_j,v)f(t,X_j,dv)\!\bigg)
        \\
        &\qquad\qquad\qquad\qquad\qquad\cdot\!\bigg(\!\varphi(X_{j},\Bar{u}^{\epsilon}_{jm_2}(t))\!-\!\!\int_{\R^4}\!\!\!\!\!\varphi(X_j,v)f(t,X_j,dv)\bigg)\!\Bigg]
        \\
        &
        \leq
        \frac{1}{M}\,L_{\varphi}^2\, C(T,b,\sigma)\,\left(1+\sup_{x\in Q}\E\big[|u(x,0)|^2\big]\right).
\end{align}

We now collect some auxiliary facts and then use these observations to complete the proof.
Since $Q\times\R^4$ is a Polish space, it embeds continuously in the compact space $[0,1]^{\N}$ endowed with the distance $\eta(x,y):=\sum_{k=1}^\infty\frac{1}{2^k}|x_k-y_k|$.
Let $\widebar{Q\times\R^4}$ denote the closure of (the image of) $Q\times\R^4$ in $[0,1]^{\N}$.
Let $U_{\eta}(Q\times\R^4)$ denote the space of function $\psi:Q\times\R^4\to\R$ which are bounded and uniformly continuous with respect to the distance $\eta$ restricted to $Q\times\R^4$.
By the continuous extension theorem we have that $U_{\eta}(Q\times\R^4)=C_b\big(\widebar{Q\times\R^4}\big)$, that is to say each bounded uniformly continuous function on $Q\times\R^4$ extends uniquely to a bounded continuous function on $\widebar{Q\times\R^4}$ and conversely each such function restricts to a bounded uniformly continuous function on $Q\times\R^4$.
The space $U_{\eta}(Q\times\R^4)$ is separable, since $\widebar{Q\times\R^4}$ is compact.
Let $\{\varphi_n\}_{n\in\N}$ be a dense countable subset and set $\psi_n:=\frac{1}{\|\varphi_n\|_{\infty}+1}\varphi_n$ for every $n$.

Given measures $\mu_j,\mu\in\pr(Q\times\R^4)$, the Portmanteau theorem implies that $\mu_j\rightharpoonup\mu$ as $j\to\infty$ if and only if $\langle\varphi,\mu_j-\mu\rangle\to0$ for every $\varphi\in U_{\eta}(Q\times\R^4)$.
Defining the distance $\delta$ on $\pr(Q\times\R^4)$ by
\begin{equation*}
    \delta(\mu,\nu):=\sum_{n\in\N}\frac{1}{2^n}|\langle\psi_n,\mu-\nu\rangle|,
\end{equation*}
we immediately see that, as $j\to\infty$,
\begin{equation*}
    \mu_j\rightharpoonup\mu \iff \delta(\mu_j,\mu)\to0.
\end{equation*}
Finally, we recall that the convergence in Wasserstein distance of order 1 is equivalent to weak convergence combined with convergence of first moments (see e.g. \cite[Chapter 7]{Villani}), i.e.
{\small
\begin{equation}\label{convergence bar f_MN vs f proof eq 2}
    \w_1(Q\times\R^4)(\mu_j,\mu)\to0\iff\begin{cases}
     \delta(\mu_j,\mu)\to0,
     \\
     \displaystyle\int_{Q\times\R^4}\!\!\!\!\!\!|x|+|u|\,d\mu_j\to\int_{Q\times\R^4}\!\!\!\!\!\!|x|+|u|\,d\mu.
    \end{cases}
\end{equation}
}

We now conclude the proof.
Using the definitions of the measures $\bar{f}_{MN}$ and $\bar{f}_N$, convexity inequalities and \eqref{convergence bar f_MN vs f proof eq 1}, we compute, for every $M,N\in\N$ and $\epsilon>0$,
\begin{align}\label{convergence bar f_MN vs f proof eq 3}
    \begin{split}
        \E\left[\delta(\bar{f}_{MN},\bar{f}_N)^2\right]
        &
        \leq\sum_{k=1}^\infty\frac{1}{2^k}\E\left[|\langle\psi_k,\bar{f}_{MN}-\bar{f}_N\rangle|^2\right]
        \\
        &\leq
        \frac{1}{M}C(T,b,\sigma)\left(1+\sup_{x\in Q}\E\big[|u(x,0)|^2\big]\right).
    \end{split}
\end{align}
Analogously we have, for every $M,N\in\N$ and $\epsilon>0$,
\begin{align}\label{convergence bar f_MN vs f proof eq 4}
    \begin{split}
        \E\left[\left|\int_{Q\times\R^4}\!\!\!\!\!\!|x|+|u|\,d\bar{f}_{MN}-\int_{Q\times\R^4}\!\!\!\!\!\!|x|+|u|\,d\bar{f}_{N}\right|^2\right]
        \leq
        \frac{1}{M}C(T,b,\sigma)\left(1+\sup_{x\in Q}\E\big[|u(x,0)|^2\big]\right).
    \end{split}
\end{align}

Given any arbitrary subsequence $M_k\to\infty$, using \eqref{convergence bar f_MN vs f proof eq 3}--\eqref{convergence bar f_MN vs f proof eq 4} and a diagonal argument, we find a sub-subsequence $M_{k_j}\to\infty$ such that
\begin{align}
    \begin{split}
        \forall N\in\N\quad \delta(\bar{f}_{NM_{k_j}},\bar{f}_N)\to0\quad\text{and}\quad\int_{Q\times\R^4}\!\!\!\!\!\!|x|+|u|\,\,\,(d\bar{f}_{NM_{k_j}}\!\!\!-d\bar{f}_N)\,\to0\quad\text{almost surely.}
    \end{split}
\end{align}
The result now follows from \eqref{convergence bar f_MN vs f proof eq 2} and the relation between almost sure convergence and convergence in probability. 
}
\end{proof}

\begin{remark}
It is possible to improve the result of the previous lemma with elementary cut-off techniques and show that
\begin{equation*}
    \lim_{M\to\infty}\E[\w_1(Q\times\R^4)(f^{\epsilon}_{MN}(t),\bar{f}_N(t))]=0\quad\text{for every $t\geq0$, $N\in\N$ and $\epsilon>0$.}
\end{equation*}
However, this method does not retain any information about the precise rate of convergence to 0, which in principle also depends on $N$. To keep track of this, we need to rely on a more sophisticated result by Fournier and Guillin \cite{Fournier2013OnTR} about the convergence of empirical laws of i.i.d particles towards their actual law.
\end{remark}

\begin{lemma}\label{rate of convergence bar f_MN vs bar f_N}
In the setting above, for any $T>0$  and any $N\in\N$,
\begin{equation}\label{rate of convergence bar f_MN vs f eq 1}
    \sup_{t\in[0,T]}\E\left[\w_1(Q\times\R^4)(\Bar{f}^{\epsilon}_{MN}(t),\bar{f}_N(t))\right]\leq C \left(1+\sup_{x\in Q}\E\left[|u_k(x,0)|^2\right]^{\frac{1}{2}}\right)\frac{1}{M^{\frac{1}{4}}},
\end{equation}
for a constant $C=C(T,b,\sigma,Q)$.
\end{lemma}
\begin{proof}
{
 The explicit expressions for $\bar{f}_{MN}^{\epsilon}$ and $\bar{f}_N$ and the convexity of the Wasserstein distance yield
\begin{gather}\label{rate of convergence bar f_MN vs bar f_N proof eq 1}
    \begin{aligned}
        \E\left[\!\w_1(Q\times\R^4)\!\!\left(f^{\epsilon}_{MN}(t),\bar{f}_N(t)\right)\right]\!
        &\leq\!
        \frac{1}{N}\!\sum_{j=1}^N\!\E\left[\w_1(Q\times\R^4)\!\!\left(\!\frac{1}{M}\!\sum_{m=1}^M\!\delta_{(X_j,\bar{u}^{\epsilon}_{jm})},\delta_{X_j}\otimes f(t,X_j,dv)\right)\!\!\right]
        \\
        &=
        \frac{1}{N}\sum_{j=1}^N\E\left[\w_1(\R^4)\left(\frac{1}{M}\sum_{m=1}^M\delta_{\bar{u}^{\epsilon}_{jm}},f(t,X_j,dv)\right)\right].
    \end{aligned}
\end{gather}
Now, observe that for each fixed $j$, the particles $\bar{u}_{jm}^{\epsilon}(t)$ for $m=1,\dots,M$ are i.i.d. with common law $f(t,X_j,dv)$.
A direct application of Theorem 1 in \cite{Fournier2013OnTR}, with $p=1$, $q=2$ and $d=4$, implies
\begin{equation} \label{rate of convergence bar f_MN vs bar f_N proof eq 2}
    \E\left[\w_1(\R^4)\left(\frac{1}{M}\sum_{m=1}^M\delta_{\bar{u}^{\epsilon}_{jm}(t)},f(t,X_j,dv)\right)\right]\leq C\, \E\left[|\bar{u}_{jm}^{\epsilon}(t)|^2\right]\, M^{-\frac{1}{4}}.
\end{equation}
We conclude using the a priori estimates \eqref{Strong existence and uniqueness for the McKean--Vlasov eq 1} and plugging formula \eqref{rate of convergence bar f_MN vs bar f_N proof eq 2} into \eqref{rate of convergence bar f_MN vs bar f_N proof eq 1}.
} 
\end{proof}

Finally we consider the last term in \eqref{splitting wasserstein distance}.
For this \emph{deterministic} term we have the following.
\begin{lemma}\label{rate of convergence bar f_N vs f}
In the setting above, for any $T\geq0$,
\begin{equation}\label{convergence bar f_MN vs f eq 1}
    \sup_{t\in[0,T]}\w_1(Q\times\R^4)(\Bar{f}_{N}(t),f(t))\leq C \frac{1}{N^{\frac{\alpha}{d}}},
\end{equation}
for a constant $C=C(T,\rho,b,\sigma,[u(\cdot,0)]_{\alpha})$.
\end{lemma}
\begin{proof}
For every $t\in[0,T]$, we consider the following pairing $\pi(t)$ defined by integration as
\begin{gather}\label{rate of convergence bar f_N vs f proof eq 1}
\begin{aligned}
    \int_{(Q\times\R^4)^2}&\varphi(x,y,u,v)\,\pi(t,dx,dy,du,dv)
    \\
    &=\sum_{j=1}^N\int_{Q_j^N}\int_{(\R^4)^2}\varphi(X_j,y,u,v)\,\pi_0(t,X_j,y,du,dv)\,dy\qquad\forall\varphi\in C_b\left((Q\times\R^4)^2\right),
\end{aligned}
\end{gather}
where $\pi_0(t,X_j,y,du,dv)$ is a chosen optimal pairing for $\w_1(\R^4)(f(t,X_j,du),f(t,y,dv))$. 
Recalling that $\meas(Q_j^N)=\nicefrac{1}{N}$, an elementary check shows that $\pi(t)$ is indeed a pairing.
Taking $\varphi(x,y,u,v)=|x-y|+|u-v|$, using the definition of $\pi(t)$ and $\pi_0(t,X_j,y)$, and recalling formula \eqref{Well-posedness of the non-linear fokker--planck equations eq 2} we compute
\begin{align*}
   \w_1(\bar{f}_N(t),f(t))
    &\leq
    \int_{(Q\times\R^4)^2}|x-y|+|u-v|\,\pi(t,dx,dy,du,dv)
    \\
    &=
    \sum_{j=1}^N\int_{Q_j^N}\int_{(\R^4)^2}|X_j-y|+|u-v|\,\pi_0(t,X_j,y,du,dv)\,dy
    \\
    &\leq
    \text{diam}(Q_j^N)+\sum_{j=1}^N\int_{Q_j^N}\int_{(\R^4)^2}|u-v|\,\pi_0(t,X_j,y,du,dv)\,dy
    \\
    &=
    \text{diam}(Q_j^N)+\sum_{j=1}^N\int_{Q_j^N}\w_1(\R^4)(f(t,X_j,du),f(t,y,dv))\,dy
    \\
    &\leq
    \text{diam}(Q_j^N)+\sum_{j=1}^N\int_{Q_j^N}C\,|X_j-y|^{\alpha}\,dy
    \\
    &\leq
    \text{diam}(Q_j^N)+ C\,\text{diam}(Q_j^N)^{\alpha},
\end{align*}
for a constant $C=C(T,\rho,b,\sigma,[u(\cdot,0)]_{\alpha})$.
Recalling that $\text{diam}(Q_j^N)=N^{-\frac{1}{d}}$ concludes the proof.
\end{proof}

\vspace{-0.5em}

\section*{Acknowledgments}
\vspace{-0.5em}
This research has been supported by the EPSRC Centre for Doctoral Training in Mathematics of Random Systems: Analysis, Modelling and Simulation (EP/S023925/1).
JAC was supported by the Advanced Grant Nonlocal-CPD (Nonlocal PDEs for Complex Particle Dynamics: Phase Transitions, Patterns and Synchronization) of the European Research Council Executive Agency (ERC) under the European Union's Horizon 2020 research and innovation programme (grant agreement No. 883363).
The authors would like to thank Lucio Galeati for pointing out a crucial mistake in the first draft of the manuscript.
\vspace{-0.4em}
\bibliography{references} 
\bibliographystyle{abbrv}

\end{document}